\documentclass[a4paper,11pt]{article}

\usepackage{amsmath}
\usepackage[latin1]{inputenc}
\usepackage{graphics}
\usepackage{verbatim}
\usepackage{listings}
\usepackage{dsfont}
\usepackage{amsfonts}
\usepackage{amssymb}
\usepackage{graphicx}
\usepackage{geometry}
\usepackage{bm}

\newtheorem{set2}{Satz}[section]
\newtheorem{theorem}[set2]{Theorem}

\newtheorem{lemma}[set2]{Lemma}
\newtheorem{notation[set2]}{Notation}

\newtheorem{proposition}[set2]{Proposition}
\newtheorem{remark}[set2]{Remark}

\newcommand{\ep}{\hfill{$\square$}}
\newenvironment{proof}[1][Proof]{\textbf{#1.} }{\
\\}

\def\XXint#1#2#3{{\setbox0=\hbox{$#1{#2#3}{\int}$}
\vcenter{\hbox{$#2#3$}}\kern-.5\wd0}}

\newcommand{\bl}{\left(}
\newcommand{\br}{\right)}
\newcommand{\e}{\varepsilon}

\newcommand{\di}{\,\mathrm{div}}
\newcommand{\DIV}{\,\mathrm{div}}
\newcommand{\R}{\mathbb{R}}
\newcommand{\N}{\mathbb{N}}
\newcommand{\C}{\mathcal}
\newcommand{\ol}{\overline}
\newcommand{\ul}{\underline}
\newcommand{\dx}{\,\mathrm dx}

\newcommand{\dt}{\,\mathrm dt}
\newcommand{\ds}{\,\mathrm ds}

\newcommand{\dxt}{\,\mathrm dx\,\mathrm dt}
\newcommand{\dxs}{\,\mathrm dx\,\mathrm ds}

\newcommand{\weaklim}{\rightharpoonup}
\newcommand{\weakstarlim}{\stackrel{\star}{\rightharpoonup}}

\newcommand{\bbC}{\mathbb{C}}
\newcommand{\bbD}{\mathbb{D}}
\newcommand{\CC}{\mathbb{C}}
\newcommand{\DD}{\mathbb{D}}

\newcommand{\uk}{u^k}

\newcommand{\pk}{p^k}
\newcommand{\pkk}{p^{k-1}}
\newcommand{\pkkk}{p^{k-2}}
\newcommand{\qk}{q^k}
\newcommand{\qkk}{q^{k-1}}
\newcommand{\chik}{\chi^k}

\newcommand{\Dt}{D_{k}}

\newcommand{\Rn}{\mathbb R^n}

\newcommand{\dU}{\dot{\C U}}
\newcommand{\dX}{\dot{\C X}}
\newcommand{\dQ}{\dot{\C Q}}
\newcommand{\db}{h}
\newcommand{\du}{\dot{u}}
\newcommand{\dchi}{\dot{\chi}}
\newcommand{\aQ}{\overline{\C Q}}
\newcommand{\aU}{\overline{\C U}}
\newcommand{\aX}{\overline{\C X}}

\newcommand{\dbb}{\mathbf{h}}
\newcommand{\dbu}{\mathbf{\dot{u}}}
\newcommand{\dbchi}{\text{\boldmath$\dot\chi$}}

\newcommand{\lu}{u^\lambda}
\newcommand{\lchi}{\chi^\lambda}
\newcommand{\ly}{y^\lambda}
\newcommand{\lz}{z^\lambda}

\newcommand{\ph}{\varphi}
\newcommand{\bxi}{\xi}

\newcommand{\sA}{\mathsf a}
\newcommand{\sB}{\mathsf b}

\newcommand{\olt}{{\ol{t}}}

\allowdisplaybreaks

\date{\today}

\author{M.~Hassan Farshbaf-Shaker$^*$,
Christian Heinemann\footnote{Weierstrass Institute for Applied Analysis and Stochastics (WIAS), Mohrenstr. 39, 10117 Berlin (Germany)}
}
\begin{document}
\title{Necessary conditions of first-order for an optimal boundary control problem for viscous damage processes in 2D}

\maketitle
\begin{abstract}
	Controlling the growth of material damage is an important
	engineering task with plenty of real world applications.
	In this paper we approach this topic
	from the mathematical point of view	by 
	investigating	an optimal boundary control problem for a damage phase-field model for
	viscoelastic media.
	We consider non-homogeneous Neumann data for the displacement field
	which describe external boundary forces and act as control variable.
	The underlying hyberbolic-parabolic PDE system for the state variables
	exhibit highly nonlinear terms which emerge in context with
	damage processes.
	The cost functional is of tracking type, and constraints for the control variable are prescribed.
	Based on recent results from \cite{FH15}, where global-in-time well-posedness of strong
	solutions to the lower 
	level problem	and existence of optimal controls of the upper level problem have been established,
	we show in this contribution
	differentiability of the control-to-state mapping,
	well-posedness of the linearization and existence of solutions of the adjoint state system.
	Due to the highly nonlinear nature of the state system
	which has by our knowledge not been considered
	for optimal control problems in the literature,
	we present a very weak formulation and estimation techniques of the associated adjoint system.
	For mathematical reasons the analysis is restricted here to the two-dimensional case.
	We conclude our results with first-order necessary optimality conditions in 
	terms of a variational inequality together with PDEs for the state and adjoint state system.
\end{abstract}
{\it AMS Subject classifications:}
%35A01,	%Existence problems: global existence, local existence, non-existence
%35A02,	%Uniqueness problems: global uniqueness, local uniqueness, non-uniqueness
35D35, %Partial differential equations, Strong solutions
%35K51, %Initial-boundary value problems for second-order parabolic systems
%35K58, %Semilinear parabolic equations
%35L53, %Initial-boundary value problems for second-order hyperbolic systems
%35L71, %Semilinear second-order hyperbolic equations
35M33, %Initial-boundary value problems for systems of mixed type
%35M87, %Systems of variational inequalities of mixed type 
35Q74, %PDEs in connection with mechanics of deformable solids
49J20, %Optimal control problems involving partial differential equations
49K20, %Optimality conditions, Problems involving partial differential equations
74A45, %Theories of fracture and damage
74D10, %Materials of strain-rate type and history type, other materials with
			%memory (including elastic materials with viscous damping, various
			%viscoelastic materials) - Nonlinear constitutive equations
74F99, %Coupling of solid mechanics with other effects
%74H20, %MECHANICS OF DEFORMABLE SOLIDS - Dynamical problems - Existence of solutions
%74H25, %MECHANICS OF DEFORMABLE SOLIDS - Dynamical problems - Uniqueness of solution
74P99; %MECHANICS OF DEFORMABLE SOLIDS - Optimization
\\[2mm]
{\it Keywords:} {optimality condition, optimal control, damage processes, phase-field model,
viscoelasticity.  \\[2mm]
}

%\selectlanguage{english}
\section{Introduction}
	Damage processes are usually highly nonlinear phenomena and their mathematical investigation
	is topic of many recent contributions in applied analysis.
	One modeling approach uses the phase-field method where a ``smooth'' variable is introduced.
%	describes the transition between the damaged and the undamaged material states.
	In the simplest case this variable is a scalar function and describes the local accumulation
	of damage in the body and the transition between the damaged and the undamaged material states.
	The popularity of phase-field models have increased in the last two decades in various 
	fields of applied mathematics, physics and engineering sciences, see \cite{Pl12}. 
	They are used to predict the micro-structure and morphological evolution of two or more 
	different phases and their mixture. 
	Specifically in damage mechanics they promise an accurate modeling and are powerful 
	techniques in prediction of material behavior. 
	For instance physical laws such as Griffith-type criteria for crack propagation
	could be encoded into a PDE/inclusion system
	and the crack paths need not to be known a priori.
	
	Despite	theses advantages the mathematical treatment of the resulting systems is challenging 
%	However further mathematical challenges emerge
	for different reasons:
	First of all the evolution law for the damage process contains a difficult type
	of nonlinear (coupling) term, a nonlinear operator acting on the time-derivative of the
	evolution variable and, depending on the type of model, even constraints on the state
	and/or its time-derivative.
	Secondly, damage processes are usually coupled with a model for elasticity where the
	material stiffness depends on the damage variable.
	In recent works this type of system has also been coupled with further
	processes such as heat conduction \cite{RR14}, phase separation
	\cite{WIAS1520}, chemical reactions \cite{KR16}, and plasticity \cite{BRRT16}.
	
	In engineering problems one is interested in prediction and even more relevant in 
	control or manipulation of damage evolution in order to prevent, for instance, complete failure
	of a structural component, see \cite{FH15} for some examples and references.

	Our main goal in this work is to provide
	a mathematical basis for optimal control problems of a time-continuous damage model
	including a first-order optimality system which has not been accomplished so far
	to the best knowledge of the authors.
	To state our problem let us
	fix an open, bounded and smooth domain $\Omega\subset\R^n$ with $n\in\{1,2\}$
	where the material is located in the reference configuration
	and let $T>0$ be a final time.
	Furthermore let $\Gamma$ be the boundary to $\Omega$ and $\nu$ the outward unit normal.
	We put $Q :=\Omega\times(0,T)$ and $\Sigma:=\Gamma\times(0,T)$.
	
	%Let $\Omega\subset\mathbb R^n,\,1\leq n\leq 2$, denote some open and bounded $C^2$-domain
	%with boundary $\Gamma$ and outward unit normal $\nu$, and let $T > 0$ be a given final-time.
	%We put $Q :=\Omega\times(0,T)$ and $\Sigma:=\Gamma\times(0,T)$, and we assume that $\lambda_T$
	%and $\lambda_\Sigma$ are given non-negative constants and
	%$\chi_T\in L^2(\Omega)$ is a given target function.
	%Moreover, we assume
	%			\begin{align*}
	%				&\chi_Q\in L^2(Q),\;\chi_T\in L^2(\Omega),\\
	%				&b_{min},\,b_{max}\in L^\infty(\Sigma)\text{ with }b_{min}\leq b_{max}\text{ a.e. in }\Sigma.
	%			\end{align*}
	We investigate the following optimal control problem:
	\begin{align}
		\textbf{(CP) }{}&\textit{Minimize the cost-functional}\notag\\
		&\quad\C J(\chi,b):=\frac{\lambda_T}{2}\|\chi(T)-\chi_T\|^2_{L^2(\Omega)}
		+\frac{\lambda_\Sigma}{2}\|b\|_{L^2(\Sigma;\Rn)}^2
		\label{functional}\\
		&\textit{subject to the hyperbolic-parabolic initial-boundary value problem }\hspace*{-3em}\phantom{\Bigg)}\notag\\
		&\quad u_{tt}-\di\big(\mathbb C(\chi)\e(u)+\mathbb D\e(u_t)\big)=\ell
			&&\text{a.e. in }Q,
			\label{eqn:elasticRegEq}\\
		&\quad\chi_t+\xi(\chi_t)-\Delta\chi_t-\Delta\chi+\frac 12 \mathbb C'(\chi)\e(u):\e(u)+f'(\chi)=0
			&&\text{a.e. in }Q,
			\label{eqn:damageRegEq}\\
		&\quad\big(\mathbb C(\chi)\e(u)+\mathbb D\e(u_t)\big)\cdot \nu  = b
			&&\text{a.e. on }\Sigma,
			\label{eqn:boundaryRegEq}\\
		&\quad\nabla(\chi+\chi_t)\cdot \nu =0
			&&\text{a.e. on }\Sigma,
			\label{eqn:boundaryRegEq2}\\
	  &\quad u(0)=u^0,\,u_t(0)=v^0,\,\chi(0)=\chi^0
		  &&\text{a.e. on }\Omega
	  	\label{eqn:initialRegEq}\\
	  &\textit{and subject to the control constraint (other types will also be allowed)}\phantom{\Bigg)}\notag\hspace*{-3em}\\
			&\quad\C B_{adm}:=\big\{b\in\C B\;|\;b_{min}\leq b\leq b_{max}\text{ a.e. in $\Sigma$ and }\|b\|_{\C B}\leq R\big\}.
			\label{Badm}
	\end{align}
	In the above problem the Banach space
	\begin{align}
		\C B:=L^2(0,T;H^{1/2}(\Gamma;\Rn))\cap H^1(0,T;L^{2}(\Gamma;\Rn))
		\label{Bspace}
	\end{align}
	is endowed with its natural norm 
	$\|\cdot\|_{\C B}:=\|\cdot\|_{L^2(0,T;H^{1/2})}+\|\cdot\|_{H^1(0,T;L^{2})}$.
	The box constraints $b_{min},\;b_{max}\in L^\infty(\Sigma)$ satisfy
	$b_{min}\leq b_{max}$ a.e. in $\Sigma$.
	Moreover
	$\chi_T$ is a given target function and
	$\lambda_T$, $\lambda_\Sigma$ and $R$ are some prescribed positive constants.

	The coupled PDE system \eqref{eqn:elasticRegEq}-\eqref{eqn:damageRegEq}
	with its initial-boundary conditions \eqref{eqn:boundaryRegEq}-\eqref{eqn:initialRegEq}
	models the lower lever problem and
	consists of the momentum balance equation \eqref{eqn:elasticRegEq}
	(according to the Kelvin-Voigt rheology)
	and a parabolic equation \eqref{eqn:initialRegEq}
	which governs the evolution of a phase-field variable $\chi$.
	The displacement field is denoted by $u$ and
	the variable $\chi$ is usually interpreted in relation with the density of micro-defects
	and therefore influences the material stiffness $\mathbb C(\cdot)$
	which is considered as a function of $\chi$.
%	\noindent
	Moreover the external volume forces are specified by $\ell$, the external 
	surface forces by $b$, the linearized strain tensor by $\e(u)=\frac12(\nabla u+(\nabla u)^T)$ 
	and
	the stress tensor by $\sigma=\mathbb C(\chi)\e(u)+\mathbb D\e(u_t)$.
	The first summand of $\sigma$ contains the elastic contribution whereas the second summand 
	models viscous effects.
	The coefficient $\mathbb C$ designates the fourth-order damage-dependent stiffness tensor 
	and $\mathbb D$ the (damage independent) viscosity tensor.			
%	The second equation \eqref{eqn:damageRegEq} specifies the parabolic evolution law for the 
%	propagation of damage described by the phase-field variable $\chi$. 

	For a mechanical motivation of system \eqref{eqn:elasticRegEq}-\eqref{eqn:initialRegEq} by 
	means of balance
	laws and constitutive relations we refer to \cite{FN96}. Global-in-time well-posedness 
	of strong solutions of the state system \eqref{eqn:elasticRegEq}-\eqref{eqn:initialRegEq} in 2D
	and existence of optimal controls for \textbf{(CP)}
	have been established in \cite{FH15}.
	For further existence, uniqueness and vanishing viscosity results
	for viscous and rate-independent damage models by making use of higher-order Laplacians
	we refer to the work \cite{KRZ11}.
	Besides these results necessary optimality conditions for \textbf{(CP)} have been left open and
	are the topic of the present paper.

	Our main result is stated in Theorem \ref{theorem:main}
	which contains necessary conditions of first-order for minimizers of \textbf{(CP)}.
	We remark that in this contribution we do not include the sub-differential
	$\partial I_{(-\infty,0]}(\chi_t)$ in the damage law \eqref{eqn:damageRegEq}
	and use $\xi(\chi_t)$ instead.
	On the one hand the incorporation of $\partial I_{(-\infty,0]}(\chi_t)$
	in the lower lever problem seems currently out of reach for necessary optimality conditions
	to \textbf{(CP)} and, on the other hand, could be approximated by $\xi(\chi_t)$.
	The nonlinear function $\xi$ is assumed to be 
	monotonically increasing with certain properties that will be fixed in the 
	next section.
	In particular $\xi$ may stem from a regularization process of the subdifferential 
	of the indicator function $I_{(-\infty,0]}$
	(e.g. Yosida- or $C^\infty$-approximations, see \cite[Chapter 5]{NT94})
	which models the so-called irreversibility condition as done in \cite{FH15}.
	Nevertheless, to the author's best knowledge,
	a nonlinearity such as $\xi(\chi_t)$ has not yet been considered in the optimal control
	literature.
	The occurrence of the time-derivative of $\chi$ in the $\xi$-nonlinearity leads to the difficult
	integral term
	$\iint \xi'(\chi_t) q \psi_t $ (where $\psi$ denotes a test-function, $q$ an adjoint variable)
	in the adjoint system.
	We propose a very weak formulation and prove an existence result for the adjoint system.
	To this end we will consider regularizations and
	derive a priori estimates by testing the system with, roughly speaking,
	time-integrated versions of the adjoint variables (see \eqref{timeIntSol} and Lemma \ref{lemmma:adjPQ}).
	A limit passage eventually yields very weak solutions to the adjoint system.

	Let us put our work into perspective.
	In contrast to modeling and analytical aspects of damage models
	the mathematical literature concerning associated optimal control is rather scarce.
	Beside the work \cite{FH15}
	we refer to \cite{WWW14} as well as to the recent preprint \cite{NWW16}
	for optimality systems of time-discretized and regularized damage phase-field models
	which follow the approach ``first time-discretize then optimize''.
	A general framework for shape optimization problems in context with semi-linear
	variational inequalities and, as an application, optimality systems for time-discretized
	damage models are explored in \cite{HS16}.	
	We also point to the work \cite{KL15} for existence
	of optimal controls for a time-continuous degenerating damage model in terms of an obstacle problem
	where the irreversibility condition has been dropped.
%	In our prior work \cite{FH15} well-posedness of the state equations
%	\eqref{eqn:elasticRegEq}-\eqref{eqn:initialRegEq}
%	and existence of optimal controls for \textbf{(CP)} have been shown.
%	Necessary optimality conditions for \textbf{(CP)} however were left open and
%	are the topic of the present paper.
	Furthermore and in opposition to a phase-field approach there is a rich literature employing
	sharp crack models with prescribed paths for optimization problems. We refer to
	\cite{HMO08, KLS12} and the references therein.
	
%	To the best of our knowledge apart from our previous paper \cite{FSH}, there are the works
%	\cite{WWW14} and \cite{KogutLeugering}.  In \cite{KogutLeugering} the authors
%	study an optimal control problem for the mixed boundary value problem for an elastic body 
%	with quasistaic evolution of an internal damage variable. Considering the Lavrentieff phenomenon,
%	they investigate this problem in the class of weak variational solutions. Existence of optimal
%	solutions are obtained, but the derivation of an optimality condition of first-order is missing
%	due to... In \cite{WWW14} an augmented Lagrangian method for the phase-field
%	approach for pressurized fractures is investigated. In constrast to In \cite{WWW14}
%	we here in this paper deal with viscoelastic effects and continue our investigations we started
%	in \cite{FSH}. To this end, we introduce in the following the model:

	\noindent\\\textbf{Structure of the paper.}\\
	The paper is organized as follows. In the next section, we list our assumptions and recall the 
	well-posedness result from \cite{FH15}, which will be the starting point for
	a deeper analysis of the solution operator.
	In Section 3 we prove differentiability of the control-to-state operator and set up the 
	linearized and adjoint problem. More precisely we establish existence of solutions to
	both systems and well-posedness to the first one.
	These intermediate results are summarized in
	Proposition \ref{prop:weakDiff},
	Proposition \ref{prop:uniqueLin}
	and Proposition \ref{prop:strongDiff}.
	At the end of this section the full optimality system is derived.
	We conclude with open problems connected to \textbf{(CP)}
	in the last section.
	
	During the course of the presented analysis we make repeated use of standard inequalities
	and embedding theorems, in particular exploiting
	%\begin{align*}
	%a\,b\leq \delta|a|^2+\underbrace{\frac{1}{4\delta}}_{C_\delta:=}|b|^2\,\quad\forall a,b\in\mathbb R\quad\forall\delta>0,
	%\end{align*}   
	continuous embeddings
	$H^2(\Omega)\subset L^\infty(\Omega)$ as well as
	$H^1(\Omega)\subset L^p(\Omega)$ for any $p\in[1,\infty)$
	valid in the two-dimensional case.
	%In particular, we have
	%\begin{align}
	%\|v\|_{L^p(\Omega)}\leq \tilde{C}_p\|v\|_{H^1(\Omega)}\qquad\forall v\in H^1(\Omega),
	%\|v\|_{L^\infty(\Omega)}\leq \tilde{C}_\infty\|v\|_{H^2(\Omega)}\qquad\forall v\in H^2(\Omega),
	%\end{align}
	%with positive constants $\tilde{C}_p$, $1\leq p\leq\infty$, that only depend on $\Omega$. 

\section{Assumptions and preliminary results}
	Let us collect the assumptions which are used throughout this work
	and restate known results obtained in \cite{FH15} concerning the state system (\ref{eqn:elasticRegEq})-(\ref{eqn:initialRegEq})
	which we extensively use for the rest of our paper.
	For well-posedness of the state system we need the following assumptions.
	\vspace*{0.4em}\\\textbf{Assumptions}
	\begin{enumerate}
	\itemindent 0.5em
		\item[\textbf{(A1)}]
			$\Omega\subseteq\R^n$ with $n\in\{1,2\}$ is assumed to be a bounded $C^2$-domain.
		\item[\textbf{(A2)}]
			The damage-dependent stiffness tensor satisfies $\mathbb C(\cdot)=\mathsf{c}(\cdot)\mathbf C$, where
			the coefficient function $\mathsf{c}$ is assumed to be of the form $\mathsf{c}=\mathsf{c}_1+\mathsf{c}_2$, where $\mathsf{c}_1\in C_{loc}^{1}(\R)$ is convex and $\mathsf{c}_2\in C_{loc}^{1}(\R)$ is concave.
			Moreover, we assume that $\mathsf{c},\mathsf{c}_1',\mathsf{c}_2'$
			(here $'$ denotes its derivative) are bounded,
			Lipschitz continuous and $\mathsf{c}(x)\geq 0$ for all $x\in\R$.
			The 4$^\mathrm{th}$ order stiffness tensor $\mathbf C\in\C L(\R_\mathrm{sym}^{n\times n};\R_\mathrm{sym}^{n\times n})$
			is assumed to be symmetric and positive definite, i.e.
			\begin{align}
				\mathbf C_{ijlk}=\mathbf C_{jilk}=\mathbf C_{lkij}\text{ and }e:\mathbf C e\geq \eta|e|^2\text{ for all }e\in \R_\mathrm{sym}^{n\times n}
					\label{eqn:Ctensor}
			\end{align}
			with constant $\eta>0$.
			%??? C_{loc}^2???		
		\item[\textbf{(A3)}]
			The function $\xi:\R\to\R$ is assumed to be Lipschitz continuous, monotonically increasing
			and $\xi(x)=0$ for $x\leq 0$.
		\item[\textbf{(A4)}]
			The 4$^\mathrm{th}$ order viscosity tensor $\mathbb D$ is given by $\mathbb D=\mu\mathbf C$ 
			and does \textbf{not} depend on the damage variable.
		\item[\textbf{(A5)}]
			The damage-dependent potential function $f$ is assumed to fulfill $f\in C_{loc}^{1}(\R)$
			and the first derivative $f'$ is Lipschitz continuous.
	\end{enumerate}
	Regarding existence of solutions of the optimal control problem \textbf{(CP)}
	we make the following additional assumptions.
	\vspace*{0.4em}\\\textbf{Assumptions}
	\begin{itemize}
	\itemindent 0.5em
		\item[\textbf{(O1)}]
			There are given non-negative constants $\lambda_T$ and $\lambda_\Sigma$.
		\item[\textbf{(O2)}]
			The target damage profile is given by
			$\chi_T\in L^2(\Omega)$.
		\item[\textbf{(O3)}]
			The admissible set of boundary controls $\C B_{adm}\subseteq \C B$
%			is given by \eqref{Badm} with
%			\begin{align*}
%				&b_{min},b_{max}\in L^\infty(\Sigma)\text{ with }b_{min}\leq b_{max}\text{ a.e. in }\Sigma.
%			\end{align*}
			is assumed to be non-empty, closed and bounded.
			$\C B$ is given by \eqref{Bspace}.
%			$\C B_{bd}$ is a nonempty, open and bounded subset of $\C B_{adm}$.
			Furthermore, let the constant $R>0$ be such that
			\begin{align*}
			\|b\|_{\C B}\leq R\quad\text{ for all }b\in \C B_{adm}.
			\end{align*}
	\end{itemize}

	\begin{remark}
		\begin{itemize}
			\item[(i)]
				In particular we may choose $\C B$ as in \eqref{Badm} which
				is then a bounded subset of $L^\infty(\Sigma)$.
			\item[(ii)]
				Note that in \cite{FH15} the assumptions for $\mathsf c_1$ and $\mathsf c_2$
				are stated as $\mathsf c_1\in C^{1,1}(\R)$ convex and $\mathsf c_2\in C^{1,1}(\R)$ concave
				with $\mathsf c,\mathsf c_1',\mathsf c_2'$ bounded and $\mathsf c\geq 0$.
				There, $C^{1,1}(\R)$ denotes the space of differentiable functions
				whose derivatives are Lipschitz continuous.
				However we do not require $\mathsf c_1$ or $\mathsf c_2$  to be bounded.
				In order to avoid confusion we choose the formulation in (A2) above.
		\end{itemize}
	\end{remark}
	For the analytical investigation of \textbf{(CP)}
	we will employ the following function spaces 
	$$
		\C Q:=\C U\times\C X\quad\subset\quad \dQ:=\dU\times\dX\quad\subset\quad \aQ:=\aU\times\aX
	$$
	with the following definitions
	\begin{align*}
	\begin{split}
		&\bullet\;\textit{space for the state system:}\\
		&\qquad\C U:=H^1(0,T;H^2(\Omega;\R^n))\cap W^{1,\infty}(0,T;H^1(\Omega;\R^n))\cap H^2(0,T;L^2(\Omega;\R^n)),\hspace*{3.3em}\\
		&\qquad\C X:=H^1(0,T;H^2(\Omega)),\\
		&\bullet\;\textit{space for the linearized state system:}\\
		&\qquad\dU:=H^1(0,T;H^1(\Omega;\Rn)\cap W^{1,\infty}(0,T;L^2(\Omega;\Rn))\cap H^2(0,T;H^1(\Omega;\Rn)^*),\\
		&\qquad\dX:=H^1(0,T;H^1(\Omega)),\\
		&\bullet\;\textit{space for the adjoint system:}\\
		&\qquad\aU:=L^2(0,T;H^1(\Omega;\Rn))\cap L^\infty(0,T;L^2(\Omega;\Rn)),\\
		&\qquad\aX:=L^2(0,T;H^1(\Omega)).
	\end{split}
	\end{align*}
	Observe that the above spaces are Banach spaces when equipped with their natural norms.
%	\vspace{5mm}
	The following results are taken from \cite[{Theorem~2.11--2.12} and Corollary~2.13--2.14]{FH15}:
	\begin{theorem}\label{wellposedness}
		Suppose that the general assumptions \textbf{(A1)}-\textbf{(A5)} are satisfied.
		Then, we have:
		\begin{itemize}
			\item[(i)]
				The state system (\ref{eqn:elasticRegEq})-(\ref{eqn:initialRegEq}) has for any $b\in\C B$,
				$\ell\in L^2(Q)$
	%			\begin{subequations}
	%			\label{PDE}
	%			\begin{align}
	%				&\int_Q u_{tt}\cdot\varphi+\mathbb C(\chi)\e(u):\e(\varphi)+\mathbb D\e(u_t):\e(\varphi)-\ell\cdot\varphi\dxt=\int_\Sigma b\cdot\varphi\dxt,\\
	%				&\int_Q \nabla\chi\cdot\nabla\psi+\nabla\chi_t\cdot\nabla\psi+\chi_t\psi+\xi(\chi_t)\psi+\frac{1}{2}\mathbb C'(\chi)\e(u):\e(u)\psi+f'(\chi)\psi\dxt=0
	%			\end{align}
	%			\end{subequations}
	%			for all $(\varphi,\psi)\in L^2(0,T;H^1(\Omega;\Rn))\times L^2(0,T;H^1(\Omega))$ and
				and initial values $u^0\in H^2(\Omega;\R^n)$, $v^0\in H^1(\Omega;\R^n)$ and $\chi^0\in H^2(\Omega)$
				a unique solution $(u,\chi)\in\C Q$ (which we call strong solution).
			\item[(ii)]
				Suppose that also \textbf{(O3)} is fulfilled
				and that $\ell$, $u^0$, $v^0$ and $\chi^0$ are fixed.
				Then there exists a positive constant $K^*_1$ (depending on $R$) such that for every
				$b\in\C B_{adm}$ the associated solution $(u,\chi)\in\C Q$ satisfies
				\begin{align}
				\label{aprioriEst}
				\|(u,\chi)\|_{\C Q}\leq K_1^*.
				\end{align}
			\item[(iii)]
				Under the assumption in (ii)
				there also exists a positive constant $K_2^*$ (depending on $R$) such that the following holds:
				Whenever $b_1,b_2\in\C B_{adm}$ are given and $(u_1,\chi_1), (u_2,\chi_2)\in\C Q$ denote the
				associated solutions of the state system, we then have
				\begin{align}
				\label{eqn:lipEst}
				\begin{aligned}
					&\|u_1-u_2\|_{\dU}
						+\|\chi_1-\chi_2\|_{\dX}
						\leq K_2^*\|b_1-b_2\|_{L^2(\Sigma)}.
				\end{aligned}
				\end{align}		
		\end{itemize}
	\end{theorem}
	
	\begin{remark}
		\begin{itemize}
			\item[(i)]
				Note that in \cite{FH15} a weaker stability estimate
				\begin{align}
				\label{eqn:lipEstRed}
				\begin{aligned}
					&\|u_1-u_2\|_{W^{1,\infty}(0,T;L^2)\cap H^1(0,T;H^1)}
						+\|\chi_1-\chi_2\|_{H^1(0,T;H^1)}
						\leq C\|b_1-b_2\|_{L^2(\Sigma)}.
				\end{aligned}
				\end{align}
				has been proven.
				To obtain \eqref{eqn:lipEst} one also needs to establish
				(here $C>0$ depends on $R$)
				\begin{align*}
					\|u_1-u_2\|_{H^2(0,T;(H^1)^*)}
						\leq C\|b_1-b_2\|_{L^2(\Sigma)}
				\end{align*}
				which follows by \eqref{eqn:lipEstRed}, \eqref{aprioriEst} and a comparison
				of the corresponding terms in \eqref{eqn:elasticRegEq}.
				In fact we find by subtraction for a.e. $t\in(0,T)$ and all $\ph\in H^1(\Omega)$:
				\begin{align*}
				\begin{split}
					&\left\langle\partial_{tt}u_1(t)-\partial_{tt}u_2(t),\ph\right\rangle_{H^1}\\
						&\quad=-\int_\Omega(\mathbb C(\chi_1(t))-\mathbb C(\chi_2(t)))\e(u_1(t)):\e(\ph)+\mathbb C(\chi_2(t))\e(u_1(t)-u_2(t)):\e(\ph)\dx\\
						&\qquad-\int_\Omega\DD\e(\partial_{t}u_1(t)-\partial_{t}u_2(t)):\e(\ph)\dx
						+\int_\Gamma (b_1(t)-b_2(t))\cdot\ph\dx
				\end{split}
				\end{align*}
				and, consequently,
				\begin{align*}
					&\|\partial_{tt}u_1-\partial_{tt}u_2\|_{L^2(0,T;(H^1)^*)}\\
						&\quad\leq C\|\chi_1-\chi_2\|_{L^2(0,T;H^1)}+C\|u_1-u_2\|_{H^1(0,T;H^1)}
							+C\|b_1-b_2\|_{L^2(\Sigma)}
						\leq \widetilde C\|b_1-b_2\|_{L^2(\Sigma)}.
				\end{align*}
			\item[(ii)]
				It follows from Theorem \ref{wellposedness}, in particular, that the control-to-state
				mapping $\C S:\C B\to \C Q$ 
				given by $S(b):=(u,\chi)$ is well defined.
				Moreover, $\C S$ is Lipschitz continuous when viewed as a mapping from the subset
				$\C B_{adm}$ of $\C B$ into the space $\dot{\C Q}$. 
		\end{itemize}
	\end{remark}	
	With a proof that resembles \cite[Theorem 3.6]{FH15} and needs no repetition here, we can 
	show the following existence result for optimal controls:
	\begin{theorem}[cf. {\cite[Theorem 3.6]{FH15}}]\label{ExistenceOptimalControl}
		Suppose that the assumptions \textbf{(A1)}-\textbf{(A5)} and \textbf{(O1)}-\textbf{(O3)}
		are fulfilled.
		Then the optimal control problem \textbf{(CP)} admits a solution
		$(\chi,b)\in\C X\times\C B_{adm}$.
	\end{theorem}
	In the present contribution we proceed with a first-order optimality system
	which will require the following enhanced differentiability assumptions
	in addition to the assumptions \textbf{(A1)}-\textbf{(A5)} and \textbf{(O1)}-\textbf{(O3)}:
	\vspace*{0.4em}\\\textbf{Assumptions}
	\begin{itemize}
	\itemindent 0.5em
		\item[\textbf{(B1)}]
			$\mathbb C(\cdot)=\mathsf{c}(\cdot)\mathbf C$
			from \textbf{(A2)} is assumed to satisfy $\mathsf c\in C_{loc}^3(\R)$;
		\item[\textbf{(B2)}]
			$\xi$ from \textbf{(A3)} is assumed to satisfy
			$\xi\in C_{loc}^2(\R)$ and $\xi''$ is bounded;
		\item[\textbf{(B3)}]
			$f$ from \textbf{(A5)} is assumed to be $f\in C_{loc}^2(\R)$;
		\item[\textbf{(B4)}]
			$\C B_{adm}$ from \textbf{(O3)} is assumed to be convex.
	\end{itemize}

\section{Analysis of a first-order optimality system to \textbf{(CP)}}
	In this section our aim is to derive a first-order optimality system to the optimal
	control problem \textbf{(CP)}.
%	At first we state a result about existence of optimal controls.
	We will prove %a well-posedness result of the corresponding linearized system
%	to \eqref{eqn:elasticRegEq}-\eqref{eqn:initialRegEq} in order to prove
	G\^{a}teaux differentiability of the solution operator and
	weak solvability of a corresponding adjoint problem.
	The latter one requires several approximation schemes and carefully designed estimations
	to handle the term $\int_Q \xi'(\chi_t)q\psi_t$ (weak form) which arises from the difficult
	non-linearity $\xi(\chi_t)$ in the state system.
	A priori estimates for the approximated system in the adjoint space $\aQ$ are derived by
	testing it with certain modified anti-derivatives with respect to time of the adjoint variables
	$p$ and $q$.
	A challenging part in the calculations is to obtain the a priori estimates globally-in-time
	on the entire interval $[0,T]$.
	Finally, at the end of this section, we will assemble the pieces and derive a first-order
	optimality system.

\subsection{Differentiability of the control-to-state mapping and the linearized state system}
	This part is devoted to the proof of G\^{a}teaux differentiability of the control-to-state
	mapping $\C S:\C B\to \dQ$.
	This endeavor is splitted into a series of intermediate results
	which we briefly describe below:
	\begin{itemize}
		\item
			Proposition \ref{prop:weakDiff}:
			By considering difference quotients of the state system in combination with a limit passage
			we prove
			existence of the linearized state system and a
			differentiability property of the control-to-state mapping $\C S:\C B\to \dQ$
			in a weak topology.
		\item
			Proposition \ref{prop:uniqueLin}:
			We establish a stability result and, consequently, well-posedness of
			the linearized problem.
		\item
			Proposition \ref{prop:strongDiff}:
			Based on the previous results we are in the position to show
			G\^{a}teaux differentiability of $\C S$ by refining the estimates for the difference 
			quotients of the state system and the linearized system.
%			Together with the stability estimate of the linearized problem we even obtain
%			Fr\'echet differentiability.
	\end{itemize}
	\begin{proposition}[Convergence to the linearized problem]
	\label{prop:weakDiff}
		Suppose that the assumptions \textbf{(A1)}-\textbf{(A5)}
		and \textbf{(B1)}-\textbf{(B3)} are fulfilled.
		Then we have:
		\begin{itemize}
			\item[(i)]
				The control-to-state mapping $\C S:\C B\to\dQ$ is differentiable in the following sense:
				\begin{align}
					\frac{\C S(b+\lambda\db)-\C S(b)}{\lambda}\weaklim (\dot u,\dot\chi)\text{ weakly-star in }\dot{\C Q}\,\text{ as }\lambda\to 0
				\label{limitWeak}
				\end{align}
				for all $b, \db\in\C B$.
			\item[(ii)]
				Furthermore the limit function $(\dot u,\dot\chi)\in\dot{\C Q}$ in (\ref{limitWeak}) is a
				weak solution of the linearized state system at $(u,\chi)=\C S(b)\in\C Q$
				in direction $h\in\C B$, i.e. $(\du,\dchi)\in\dQ$ fulfills
				\begin{align}
					&
					\begin{aligned}
						&\int_0^T\langle \du_{tt},\varphi\rangle_{H^1}\dt+\int_Q \mathbb C'(\chi)\dchi\e(u):\e(\varphi)+\mathbb C(\chi)\e(\du):\e(\varphi)+\mathbb D\e(\du_t):\e(\varphi)\dxt\\
						&\quad=\int_\Sigma \db\cdot\varphi\dxt,
					\end{aligned}
					\label{linUPDE}
					\\
					&
					\begin{aligned}
						&\int_Q \nabla\dchi\cdot\nabla\psi+\nabla\dchi_t\cdot\nabla\psi+\dchi_t\psi+\xi'(\chi_t)\dchi_t\psi+\frac{1}{2}\mathbb C''(\chi)\dchi\e(u):\e(u)\psi\dxt\\
						&\quad+\int_Q \mathbb C'(\chi)\e(\du):\e(u)\psi+f''(\chi)\dchi\psi\dxt=0
					\end{aligned}
					\label{linChiPDE}
				\end{align}
		for all $(\varphi,\psi)\in \aQ$ and
		with initial values $u(0)=u_t(0)=\chi(0)=0$.
		\end{itemize}
	\end{proposition}
	\begin{proof}
		Let $\lambda\in\R$ and $b, b+\lambda\db\in\C B$ and define
		$(u,\chi):=\C S(b)$ and $(u^\lambda,\chi^\lambda):=\C S(b+\lambda\db)$.
		By Theorem \ref{wellposedness} (ii)-(iii) there exist positive constants $K_1^*$ and $K_2^*$ 
		(depending on $R$) such that 
		\begin{align}
			&\left\|u^\lambda\right\|_{\C U}+\left\|\chi^\lambda\right\|_{\C X} \leq K_1^*,
			&&\left\|\frac{u^\lambda-u}{\lambda}\right\|_{\dU}+\left\|\frac{\chi^\lambda-\chi}{\lambda}\right\|_{\dX}
				\leq K_2^*\|h\|_{L^2(0,T;L^2)}.
			\label{linBound}
		\end{align}
		Therefore the sequence 
		$\{(\frac{u^{\lambda}-u}{\lambda},\frac{\chi^{\lambda}-\chi}{\lambda})\}$ 
		is uniformly bounded in $\dU\times\dX$ with respect to $\lambda$ and we may
		extract a weakly convergent subsequence and obtain by omitting the subscript
		\begin{align}\label{limitWeakProof}
		\left(\frac{u^{\lambda}-u}{\lambda},\frac{\chi^{\lambda}-\chi}{\lambda}\right)
			\rightarrow (\dot u,\dot\chi)\text{ weakly-star in }\dot{\C Q}\,\text{ as }\lambda\to 0. 
		\end{align}
		In the next step we are going to show that $(\dot u,\dot\chi)\in\dQ$ is a
		weak solution of the linearized system \eqref{linUPDE}-\eqref{linChiPDE} at
		$(u,\chi)=\C S(b)\in \C Q$ in direction $h$.
		We sketch the passage to the limit for the nonlinear terms.
		To this end we prove the following convergence statements as $\lambda\to 0$:
		\begin{align*}
			\text{(a)}\quad&\int_Q\frac{\mathbb C(\chi^\lambda)\e(u^\lambda)-\mathbb C(\chi)\e(u)}{\lambda}:\e(\varphi)\dxt\rightarrow\int_Q \mathbb C'(\chi)\dchi\e(u):\e(\varphi)+\mathbb C(\chi)\e(\du):\e(\varphi)\dxt,\\
			\text{(b)}\quad&\int_Q\frac{1}{2}\frac{\mathbb C'(\chi^\lambda)\e(u^\lambda):\e(u^\lambda)-\mathbb C'(\chi)\e(u):\e(u)}{\lambda}\psi\dxt\\
				&\qquad\qquad\qquad\rightarrow\int_Q\frac{1}{2}\mathbb C''(\chi)\dchi\e(u):\e(u)+\mathbb C'(\chi)\e(\du):\e(u)\psi\dxt,\\
			\text{(c)}\quad&\int_Q \frac{f'(\chi^\lambda)-f'(\chi)}{\lambda}\psi\dxt\rightarrow\int_Q f''(\chi)\dchi\psi\dxt,\\
			\text{(d)}\quad&\int_Q \left(\frac{\xi(\chi^\lambda_t)-\xi(\chi_t)}{\lambda}\right)\psi\dxt\rightarrow \int_Q \xi'(\chi_t)\dchi_t\psi\dxt.
		\end{align*}
		and test-functions $(\varphi,\psi)\in \aQ$.
		%\vspace*{0.9cm}
		\begin{itemize}
			\item[]\hspace*{-2em}To (a):
				We apply the following splitting
				\begin{align*}
					&\int_Q \frac{\mathbb C(\chi^\lambda)\e(u^\lambda)-\mathbb C(\chi)\e(u)}{\lambda}:\e(\ph)\dxt\\
						&\quad=\underbrace{\int_Q\frac{\mathbb C(\chi^\lambda)-\mathbb C(\chi)}{\lambda}\e(u):\e(\ph)\dxt}_{=:T_1}
							+\underbrace{\int_Q\mathbb C(\chi^\lambda)\e\Big(\frac{u^\lambda-u}{\lambda}\Big):\e(\ph)\dxt}_{=:T_2}.
				\end{align*}
				The first term is treated by
				the mean value theorem applied to each tensor component (applicable due to Assumption \textbf{(A2)}),
				e.g.
				$$
					\frac{\mathbb C_{ijkl}(\chi^\lambda)-\mathbb C_{ijkl}(\chi)}{\lambda}=\mathbb C_{ijkl}'(\chi+\bar{\lambda}_{ijkl}(\chi^\lambda-\chi))\left(\frac{\chi^\lambda-\chi}{\lambda}\right)
				$$
				for values $\bar{\lambda}_{ijkl}\in[0,\lambda]$,
				\begin{align*}
					T_1=\int_Q\{\mathbb C_{ijkl}'(\chi+\bar{\lambda}_{ijkl}(\chi^\lambda-\chi))\}_{0\leq i,j,k,l\leq 1}\left(\frac{\chi^\lambda-\chi}{\lambda}\right)\e(u):\e(\varphi)\dxt
				\end{align*}
				By using the dominated convergence theorem of Lebesgue
				and the a priori estimates in \eqref{linBound}, we obtain for a subsequence
				\begin{align*}
					&T_1\to\int_Q \mathbb C'(\chi)\dchi\e(u):\e(\varphi)\dxt,
					&&T_2\to\int_Q \mathbb C(\chi)\e(\du):\e(\varphi)\dxt.
				\end{align*}
			\item[]\hspace*{-2em}To (b):
				Via splitting and similar arguments we obtain for a subsequence $\lambda\to 0$
				\begin{align*}
					&\int_Q\frac{1}{2}\frac{\mathbb C'(\chi^\lambda)\e(u^\lambda):\e(u^\lambda)-\mathbb C'(\chi)\e(u):\e(u)}{\lambda}\psi\dxt\\
					&=\underbrace{\int_Q\frac{1}{2}\frac{\mathbb C'(\chi^\lambda)-\mathbb C'(\chi)}{\lambda}\e(u):\e(u)\psi\dxt}_{
							\to\int_Q\frac{1}{2}\mathbb C''(\chi)\dchi\e(u):\e(u)\psi\dxt}
						+\underbrace{\int_Q\mathbb C'(\chi^\lambda)\e\left(\frac{u^\lambda-u}{\lambda}\right):\e\left(\frac{u^\lambda+u}{2}\right)\psi\dxt}_{
							\to\int_Q\mathbb C'(\chi)\e(\du):\e(u)\psi\dxt}.
				\end{align*}
			\item[]\hspace*{-2em}To (c)-(d):
				These properties follow by similar arguments with less effort.
		\end{itemize}
		\ep
	\end{proof}

	\begin{proposition}[Stability of the linearized problem]
	\label{prop:uniqueLin}
		Let $b\in\C B$ be given and denote $(u,\chi)=\C S(b)$.
		Furthermore, let $h_1, h_2\in\C B$ be two given directions.
		Then for two weak solutions $(\dot u_1,\dot \chi_1)$ and $(\dot u_2,\dot \chi_2)$
		of the linearized system \eqref{linUPDE}-\eqref{linChiPDE}
		to the associated directions $h_1$ and $h_2$, respectively,
		we have the Lipschitz estimate
		\begin{align*}
			&\|\dot u_1-\dot u_2\|_{H^1(0,T;H^1)\cap W^{1,\infty}(0,T;L^2)}
				+\|\dot\chi_1-\dot\chi_2\|_{H^1(0,T;H^1)}
				\leq C\|h_1-h_2\|_{L^2(0,T;L^2(\Gamma;\R^n))}.
		\end{align*}
%		Then $\dot u_1=\dot u_2$ and $\dot \chi_1=\dot \chi_2$.
	\end{proposition}
	\begin{proof}
		Introducing the differences
		\begin{align*}
			\dbu=\dot u_1-\dot u_2\quad\text{and}\quad
			\dbchi=\dot \chi_1-\dot \chi_2\quad\text{and}\quad
			\dbb=h_1-h_2.
		\end{align*}
		and subtracting equations \eqref{linUPDE}-\eqref{linChiPDE}
		with the corresponding solutions from each other we obtain by 
		testing with $(\dbu_t,\dbchi_t)$
		\begin{align*}
			&\int_0^t\langle \dbu_{tt},\dbu_t\rangle_{H^1}\ds+\int_0^t\int_\Omega\mathbb C'(\chi)\dbchi\e(u):\e(\dbu_t)+\mathbb C(\chi)\e(\dbu):\e(\dbu_t)+\mathbb D\e(\dbu_t):\e(\dbu_t)\dxs\\
			&\quad=\int_0^t\int_\Gamma \dbb\cdot\dbu_t\dxs,\\
			&\int_0^t\int_\Omega \nabla\dbchi\cdot\nabla\dbchi_t+|\nabla\dbchi_t|^2+|\dbchi_t|^2+\xi'(\chi_t)|\dbchi_t|^2+\frac{1}{2}\mathbb C''(\chi)\e(u):\e(u)\dbchi\dbchi_t\dxs\\
			&\quad+\int_0^t\int_\Omega \mathbb C'(\chi)\e(\dbu):\e(u)\dbchi_t+f''(\chi)\dbchi\dbchi_t\dxs=0.
		\end{align*}
		We estimate the following terms occurring on the left-hand side of these equations by making use of Young's and H\"older's inequalities, standard Sobolev embeddings,
		the regularity of the state variables $(u,\chi)$ and the properties of the $\beta$-regularization:
		\begin{align*}
			\Big|\int_0^t\int_\Omega \mathbb C'(\chi)\dbchi\e(u):\e(\dbu_t)\dxt\Big|
				\leq{}& \|\mathbb C'(\chi)\|_{L^\infty(L^\infty)}\|\dbchi\|_{L^2(0,t;L^4)}\|\e(u)\|_{L^\infty(L^4)}\|\e(\dbu_t)\|_{L^2(0,t;L^2)}\\
				\leq{}& \delta\|\e(\dbu_t)\|_{L^2(0,t;L^2)}^2 + C_\delta\|\dbchi\|_{L^2(0,t;H^1)}^2,\\
			\Big|\int_0^t\int_\Omega \mathbb C(\chi)\e(\dbu):\e(\dbu_t)\dxs\Big|
				\leq{}&\delta\|\e(\dbu_t)\|_{L^2(0,t;L^2)}^2 + C_\delta\|\e(\dbu)\|_{L^2(0,t;L^2)}^2,\\
			\int_0^t\int_\Omega \xi'(\chi_t)|\dbchi_t|^2\dxs
				\geq{}&0,\\
			\Big|\int_0^t\int_\Omega \mathbb C''(\chi)\e(u):\e(u)\dbchi\dbchi_t\dxs\Big|
				\leq{}&\|\mathbb C''(\chi)\|_{L^\infty(L^\infty)}\|\e(u)\|_{L^\infty(L^4)}^2\|\dbchi\|_{L^2(0,t;L^4)}\|\dbchi_t\|_{L^2(0,t;L^4)}\\
				\leq{}&\delta\|\dbchi_t\|_{L^2(0,t;H^1)}^2+C_\delta\|\dbchi\|_{L^2(0,t;H^1)}^2,\\
			\Big|\int_0^t\int_\Omega \mathbb C'(\chi)\e(\dbu):\e(u)\dbchi_t\dxs\Big|
				\leq{}&\|\mathbb C'(\chi)\|_{L^\infty(L^\infty)}\|\e(u)\|_{L^\infty(L^4)}\|\e(\dbu)\|_{L^2(0,t;L^2)}\|\dbchi_t\|_{L^2(0,t;L^4)}\\
				\leq{}&\delta\|\dbchi_t\|_{L^2(0,t;H^1)}^2+C_\delta\|\e(\dbu)\|_{L^2(0,t;L^2)}^2,\\
			\Big|\int_0^t\int_\Omega f''(\chi)\dbchi\dbchi_t\dxs\Big|
				\leq{}&\|f''(\chi)\|_{L^\infty(L^\infty)}\|\dbchi\|_{L^2(0,t;L^2)}^2\|\dbchi_t\|_{L^2(0,t;L^2)}^2\\
				\leq{}&\delta\|\dbchi_t\|_{L^2(0,t;L^2)}^2+C_\delta\|\dbchi\|_{L^2(0,t;L^2)}^2.
		\end{align*}
		The right-hand side of the first equation is treated by Young's inequality and
		the trace theorem via
		\begin{align*}
			\int_0^t\int_\Gamma \dbb\cdot\dbu_t\dxs
				\leq \delta\|\dbu_t\|_{L^2(0,t;H^1(\Omega;\R^n))}^2
					+C_\delta\|\dbb\|_{L^2(0,t;L^2(\Gamma;\R^n))}^2.
		\end{align*}
		Moreover, by making use of the embedding $H^1(\Omega;\R^n)\hookrightarrow H^1(\Omega;\R^n)^*$
		given by $u\mapsto (u,\cdot)_{L^2}$, we find
		\begin{align*}
			\frac12\|\dbu_t(t)\|_{L^2}^2
				=\int_0^t\frac{d}{dt}\frac12\langle\dbu_t(s),\dbu_t(s)\rangle_{(H^1)^*\times H^1}\ds
				=\int_0^t\langle \dbu_{tt}(s),\dbu_t(s)\rangle_{(H^1)^*\times H^1}\ds.
		\end{align*}
		Applying these calculations and adding the equations above, we obtain
		\begin{align*}
			&\|\dbu_t(t)\|_{L^2}^2+\|\nabla\dbchi(t)\|_{L^2}^2+\|\e(\dbu_t)\|_{L^2(0,t;L^2)}^2
				+\|\dbchi_t\|_{L^2(0,t;H^1)}^2\\
				&\quad\leq\delta\Big(\|\dbu_t\|_{L^2(0,t;H^1)}^2+\|\dbchi_t\|_{L^2(0,t;H^1)}^2\Big)
					+C_\delta\Big(\|\dbu\|_{L^2(0,t;L^2)}^2+\|\dbchi\|_{L^2(0,t;H^1)}^2+\|\dbb\|_{L^2(0,t;L^2(\Gamma;\R^n))}^2\Big).
		\end{align*}
		Now, adding
		$\|\dbu_t\|_{L^2(0,t;L^2)}^2+\|\dbu\|_{L^2(0,t;H^1)}^2+\|\dbchi\|_{L^2(0,t;H^1)}^2$
		on both sides, applying Korn's inequality and choosing $\delta>0$ small enough, we find
		\begin{align*}
			&\|\dbu_t(t)\|_{L^2}^2+\|\nabla\dbchi(t)\|_{L^2}^2+\|\dbu\|_{H^1(0,t;H^1)}^2
				+\|\dbchi\|_{H^1 (0,t;H^1)}^2\\
				&\quad\leq C\Big(\|\dbu_t\|_{L^2(0,t;L^2)}^2+\|\dbu\|_{L^2(0,t;H^1)}^2+\|\dbchi\|_{L^2(0,t;H^1)}^2+\|\dbb\|_{L^2(0,t;L^2(\Gamma;\R^n))}^2\Big).
		\end{align*}
		This yields with the help of the estimates
		\begin{align*}
			&\|\dbu\|_{L^2(0,t;H^1)}^2\leq C\int_0^t\|\dbu_t\|_{L^2(0,s;H^1)}^2\ds,\\
			&\|\dbchi\|_{L^2(0,t;H^1)}^2\leq C\int_0^t\|\dbchi_t\|_{L^2(0,s;H^1)}^2\ds,\\
			&\|\dbchi(t)\|_{L^2}^2\leq C\|\dbchi_t\|_{L^2(0,t;L^2)}^2\ds,
		\end{align*}
		the following inequality
		\begin{align*}
			&\|\dbu_t(t)\|_{L^2}^2+\|\dbchi(t)\|_{H^1}^2+\|\dbu\|_{H^1(0,t;H^1)}^2
				+\|\dbchi\|_{H^1(0,t;H^1)}^2\\
				&\quad\leq C\|\dbb\|_{L^2(0,T;L^2(\Gamma;\R^n))}^2+C\int_0^t\Big(\|\dbu_t\|_{L^2}^2+\|\dbu_t\|_{L^2(0,s;H^1)}^2+\|\dbchi_t\|_{L^2(0,s;H^1)}^2+\|\dbchi\|_{H^1}^2\Big)\ds.
		\end{align*}
		Gronwall's lemma shows the claim.
		\ep
	\end{proof}

%\subsection{Strong differentiability of the solution operator}
	\begin{proposition}[Strong differentiability of the control-to-state mapping]
		\label{prop:strongDiff}
		Under\linebreak the assumptions of Proposition \ref{prop:weakDiff}
		the convergence \eqref{limitWeak} is even strong in $\dQ$.
		Moreover, the operator $\C S:\C B\to\dQ$ is G\^{a}teaux differentiable
		and we have $\langle D\C S(b),h\rangle=(\du,\dchi)$.
	\end{proposition}
	\begin{proof}
		Let $\lambda\in\R$ and $b, b+\lambda\db\in\C B$ and define
		$$
			(u^\lambda,\chi^\lambda):=\C S(b+\lambda\db),\qquad
			(u,\chi):=\C S(b).
		$$
		Furthermore, let $(\du,\dchi)$ be the unique solution of the linearized system
		\eqref{linUPDE}-\eqref{linChiPDE} at $b$ in direction $\db$.
		We consider the following system arising from the calculations
		\begin{align*}
			\Big[\text{PDE system \eqref{eqn:elasticRegEq}-\eqref{eqn:initialRegEq} for $b+\lambda\db$}\Big]
				\;\;&-\;\;
				\Big[\text{PDE system \eqref{eqn:elasticRegEq}-\eqref{eqn:initialRegEq} for $b$}\Big]\\
			&-\lambda\times\Big[\text{PDE system \eqref{linUPDE}-\eqref{linChiPDE} for $(b,\db)$}\Big].
		\end{align*}
		By introducing the functions $(\ly,\lz)\in\dQ$ as
		$$
			\ly:=u^\lambda-u-\lambda\du\qquad
			\lz:=\chi^\lambda-\chi-\lambda\dchi,
		$$
		the resulting system can be written as
		\begin{align}
			&\left\{
			\begin{aligned}
				&\int_0^t \langle \ly_{tt},\ph\rangle_{H^1}\ds
					+\int_0^t\int_\Omega \bl\bbC(\lchi)-\bbC(\chi)-\lambda\dchi\bbC'(\chi)\br\e(u):\e(\ph)\dxs\\
				&\quad+\int_0^t\int_\Omega \bbC(\lchi)\e(\ly):\e(\ph)
					+\bbD\e(\ly_t):\e(\ph)\dxs=0,
			\end{aligned}
			\right.
			\label{strPDE1}\\
			&\left\{
			\begin{aligned}
				&\int_0^t\int_\Omega\nabla \lz\cdot\nabla\psi+\nabla\lz_t\cdot\nabla\psi+\lz_t\psi\dxs
					+\int_0^t\int_\Omega\bl\bxi(\lchi_t)-\bxi(\chi_t)-\lambda\dchi_t\bxi'(\chi_t)\br\psi\dxs\\
				&\quad+\int_0^t\int_\Omega\frac12\bl\bbC'(\lchi)\e(\lu):\e(\lu)-\bbC'(\chi)\e(\lu):\e(\lu)
					-\lambda\dchi\bbC''(\chi)\e(u):\e(u)\br\psi\dxs\\
				&\quad+\int_0^t\int_\Omega\frac12\bl\bbC'(\chi)\e(\lu):\e(\lu)-\bbC'(\chi)\e(u):\e(\lu)
					-\lambda\bbC'(\chi)\e(\du):\e(u)\br\psi\dxs\\
				&\quad+\int_0^t\int_\Omega\frac12\bbC'(\chi)\e(u):\e(\ly)\psi\dxs
					+\int_0^t\int_\Omega\bl f'(\lchi)-f'(\chi)-\lambda\dchi f''(\chi)\br\psi\dxs\\
				&\quad=0.
			\end{aligned}
			\right.
			\label{strPDE2}
		\end{align}
%		By convexity of $\bxi$ we also find
%		\begin{align*}
%			\bxi(\lchi_t)-\bxi(\chi_t)-\lambda\dchi_t\bxi'(\chi_t)
%				\geq{}& \bxi'(\chi_t)\lz_t.
%		\end{align*}
		Now, testing \eqref{strPDE1} with $\ph=\ly_t$
		and \eqref{strPDE2} with $\psi=\lz_t$
		and adding both equations, we find
		\begin{align*}
				&\int_0^t \langle \ly_{tt},\ly_t\rangle_{H^1}\ds
					+\int_0^t\int_\Omega\bbD\e(\ly_t):\e(\ly_t)
					+\nabla \lz\cdot\nabla\lz_t	+ |\nabla\lz_t|^2+|\lz_t|^2\dxs\\
				&\quad=
					\underbrace{-\int_0^t\int_\Omega \bl\bbC(\lchi)-\bbC(\chi)-\lambda\dchi\bbC'(\chi)\br\e(u):\e(\ly_t)\dxs}_{=:T_1},\\
					&\qquad\underbrace{-\int_0^t\int_\Omega\bbC(\lchi)\e(\ly):\e(\ly_t)\dxs}_{=:T_2}
					\underbrace{-\int_0^t\int_\Omega\bl\bxi(\lchi_t)-\bxi(\chi_t)-\lambda\dchi_t\bxi'(\chi_t)\br\lz_t\dxs}_{=:T_3}\\
					&\qquad\underbrace{-\int_0^t\int_\Omega\frac12\bl\bbC'(\lchi)\e(\lu):\e(\lu)-\bbC'(\chi)\e(\lu):\e(\lu)
					-\lambda\dchi\bbC''(\chi)\e(u):\e(u)\br\lz_t\dxs}_{=:T_4}\\
					&\qquad\underbrace{-\int_0^t\int_\Omega\frac12\bl\bbC'(\chi)\e(\lu):\e(\lu)-\bbC'(\chi)\e(u):\e(\lu)
					-\lambda\bbC'(\chi)\e(\du):\e(u)\br\lz_t\dxs}_{T_5}\\
					&\qquad\underbrace{-\int_0^t\int_\Omega\frac12\bl\bbC'(\chi)\e(u):\e(\lu)-\bbC'(\chi)\e(u):\e(u)
					-\lambda\bbC'(\chi)\e(\du):\e(u)\br\lz_t\dxs}_{T_6}\\
					&\qquad\underbrace{-\int_0^t\int_\Omega\frac12\bbC'(\chi)\e(u):\e(\ly)\lz_t\dxs}_{=:T_7}
					\underbrace{-\int_0^t\int_\Omega\bl f'(\lchi)-f'(\chi)-\lambda\dchi f''(\chi)\br\lz_t\dxs}_{=:T_8}.
		\end{align*}
		To proceed, we make use of the following estimates:
		By using Taylor's theorem and boundedness
		\begin{align}
		\label{chiBound}
			\|\chi\|_{L^\infty(Q)}+\|\lchi\|_{L^\infty(Q)}\leq L
		\end{align}
		with respect to $\lambda$ (applying Theorem \ref{wellposedness} (ii))
		as well as $\|\bxi'\|_{L^\infty}+\|\bxi''\|_{L^\infty}<+\infty$
		and assumptions (B1)-(B3),
		we obtain the following estimates
		\begin{align*}
				\Big|\bbC(\lchi)-\bbC(\chi)-\lambda\dchi\bbC'(\chi)\Big|
					\leq{}& \sup_{|x|<L}|\bbC'(x)||\lz|+C\sup_{|x|<L}|\bbC''(x)||\lchi-\chi|^2\\
					\leq{}& C(|\lz|+|\lchi-\chi|^2),\\
				\Big|\bbC'(\lchi)-\bbC'(\chi)-\lambda\dchi\bbC''(\chi)\Big|
					\leq{}& \sup_{|x|<L}|\bbC''(x)||\lz|+C\sup_{|x|<L}|\bbC'''(x)||\lchi-\chi|^2\\
					\leq{}& C(|\lz|+|\lchi-\chi|^2),\\
				\Big|\xi(\lchi)-\xi(\chi)-\lambda\dchi \xi'(\chi)\Big|
					\leq{}& C(|\lz|+|\lchi-\chi|^2),\\
				\Big|f'(\lchi)-f'(\chi)-\lambda\dchi f''(\chi)\Big|
					\leq{}& C(|\lz|+|\lchi-\chi|^2).
		\end{align*}
		The above estimates, the a priori estimates
		$$
			\|(u^\lambda,\chi^\lambda)\|_{\C Q}\leq C
		$$		
		(in particular\eqref{chiBound})
		and a special vector-valued version of the Gagliardo-Nirenberg inequality in 2D
		$$
			\|w\|_{L^4}\leq C\|w\|_{H^1}^{1/2}\|w\|_{L^2}^{1/2}\quad\text{for all }w\in H^1(\Omega;\R^m)
		$$		
		allow us to treat the terms $T_1$, $T_2$, $T_4$, $T_5$, $T_6$, $T_7$ and $T_8$ as follows
		\begin{align*}
			T_1\leq{}&C\int_0^t\int_\Omega|\lz||\e(u)||\e(\ly_t)|+|\lchi-\chi|^2|\e(u)||\e(\ly_t)|\dxs\\
				\leq{}&C\bl\|\lz\|_{L^2(0,t;L^3)}+\|\lchi-\chi\|_{L^2(0,t;L^6)}^2\br\|\e(u)\|_{L^\infty(L^6)}\|\e(\ly_t)\|_{L^2(0,t;L^2)}\\
				\leq{}&\delta\|\e(\ly_t)\|_{L^2(0,t;L^2)}^2+C_\delta\|\lz\|_{L^2(0,t;H^1)}^2+C_\delta\|\lchi-\chi\|_{L^2(0,t;H^1)}^4\\
			T_2\leq{}&\delta\|\e(\ly_t)\|_{L^2(0,t;L^2)}^2+C_\delta\|\e(\ly)\|_{L^2(0,t;L^2)}^2,\\
			T_4={}&-\int_0^t\int_\Omega\frac12\bl\bbC'(\lchi)-\bbC'(\chi)
					-\lambda\dchi\bbC''(\chi)\br\e(\lu):\e(\lu)\lz_t\dxs\\
					&-\int_0^t\int_\Omega\frac12\lambda\dchi\bbC''(\chi)\e(\lu+u):\e(\lu-u)\lz_t\dxs\\
				\leq{}&C\int_0^t\int_\Omega|\lz||\e(\lu)|^2|\lz_t|\dxs
					+C\int_0^t\int_\Omega|\lchi-\chi|^2|\e(\lu)|^2|\lz_t|\dxs\\
					&+|\lambda| C\int_0^t\int_\Omega|\dchi||\e(\lu+u)||\e(\lu-u)||\lz_t|\dxs\\
				\leq{}&\delta\|\lz_t\|_{L^2(0,t;H^1)}^2+C_\delta\|\e(\lu)\|_{L^\infty(L^4)}^2\|\lz\|_{L^2(0,t;H^1)}^2\\
					&+C\|\e(\lu)\|_{L^\infty(L^5)}^2\|\lchi-\chi\|_{L^\infty(L^5)}^2\|\lz_t\|_{L^1(0,t;L^5)}\\
					&+|\lambda| C\|\dchi\|_{L^\infty(L^6)}\|\e(\lu+u)\|_{L^\infty(L^6)}\|\e(\lu-u)\|_{L^2(L^2)}\|\lz_t\|_{L^2(0,t;L^6)}\\
				\leq{}&\delta\|\lz_t\|_{L^2(0,t;H^1)}^2
					+C_\delta\|\lz\|_{L^2(0,t;H^1)}^2
					+C_\delta\|\lchi-\chi\|_{H^1(H^1)}^4
					+|\lambda|^2 C_\delta\|\e(\lu-u)\|_{L^2(L^2)}^2,\\
			T_5={}&-\frac12\int_0^t\int_\Omega\bbC'(\chi)\e(\ly):\e(\lu)\lz_t\dxs
					-\frac\lambda2\int_0^t\int_\Omega\bbC'(\chi)\e(\du):\e(\lu-u)\lz_t\dxs\\
				\leq{}&C\|\e(\ly)\|_{L^2(0,t;L^2)}\|\e(\lu)\|_{L^\infty(L^4)}\|\lz_t\|_{L^2(0,t;L^4)}\\
					&+|\lambda| C\int_0^t\|\e(\du)\|_{L^2}\|\e(\lu-u)\|_{L^4}\|\lz_t\|_{L^4}\ds\quad\text{(apply Gagliardo-Nirenberg)}\\
				\leq{}&\delta\|\lz_t\|_{L^2(0,t;H^1)}^2+C_\delta\|\e(\ly)\|_{L^2(0,t;L^2)}^2\\
					&+|\lambda| C\|\e(\du)\|_{L^\infty(L^2)}\int_0^t\|\e(\lu-u)\|_{H^1}^{1/2}\|\e(\lu-u)\|_{L^2}^{1/2}\|\lz_t\|_{L^4}\ds\\
				\leq{}&\delta\|\lz_t\|_{L^2(0,t;H^1)}^2+C_\delta\|\e(\ly)\|_{L^2(0,t;L^2)}^2
					+|\lambda|^2 C_\delta\|\e(\lu-u)\|_{L^2(L^2)},\\
			T_6={}&-\frac12\int_0^t\int_\Omega\bbC'(\chi)\e(\ly):\e(u)\lz_t\dxs\\
				\leq{}&C\|\e(\ly)\|_{L^2(0,t;L^2)}\|\e(\lu)\|_{L^\infty(L^4)}\|\lz_t\|_{L^2(0,t;L^4)}\\
				\leq{}&\delta\|\lz_t\|_{L^2(0,t;H^1)}^2+C_\delta\|\e(\ly)\|_{L^2(0,t;L^2)}^2\\
			T_7\leq{}& \delta\|\lz_t\|_{L^2(0,t;H^1)}^2+C_\delta\|\e(\ly)\|_{L^2(0,t;L^2)}^2,\\
			T_8\leq{}&\delta\|\lz_t\|_{L^2(0,t;L^2)}^2+C_\delta\|\lchi-\chi\|_{H^1(H^1)}^4+C_\delta\|\lz\|_{L^2(0,t;L^2)}^2.
		\end{align*}
		Due to the low time-regularity of the damage variables, the term $T_3$ needs to
		be treated differently.
		To this end, we find by the mean value theorem with
		$$
			\mu\in[\min\{\lchi_t,\chi_t\}, \max\{\lchi_t,\chi_t\}] \text{ suitably choosen,}
		$$
		Young's inequality, $\bxi'\geq 0$
		and the monotonicity of $\bxi'$ that
		\begin{align*}
			-\bl\bxi(\lchi_t)-\bxi(\chi_t)-\lambda\dchi_t\bxi'(\chi_t)\br\lz_t
				={}&-\bl\bxi'(\mu)(\lchi_t-\chi_t)-\lambda\dchi_t\bxi'(\chi_t)\br\lz_t\\
				={}&-\bxi'(\mu)(\lz_t+\lambda\dchi_t)\lz_t+\lambda\dchi_t\bxi'(\chi_t)\lz_t\\
				\leq{}&-\bxi'(\mu)\lambda\dchi_t\lz_t+\lambda\dchi_t\bxi'(\chi_t)\lz_t\\
				={}&-(\bxi'(\mu)-\bxi'(\chi_t))\lambda\dchi_t\lz_t\\
				\leq{}&\delta|\lz_t|^2+|\lambda|^2C_\delta|\bxi'(\mu)-\bxi'(\chi_t)|^2|\dchi_t|^2\\
				\leq{}&\delta|\lz_t|^2+|\lambda|^2C_\delta|\bxi'(\lchi_t)-\bxi'(\chi_t)|^2|\dchi_t|^2.
		\end{align*}
		For further considerations we define $f_\lambda\in L^\infty(Q)$ by
		$$
			f_\lambda:=|\bxi'(\lchi_t)-\bxi'(\chi_t)|^2
		$$
		and thus obtain
		$$
			T_3\leq 
				\delta\|\lz_t\|_{L^2(0,t;L^2)}^2+|\lambda|^2C_\delta\int_0^t\int_\Omega f_\lambda |\dchi_t|^2\dxs.
		$$
		Note that due to continuity of the solution operator $\C S:\C B\to\dQ$
		by Theorem \ref{wellposedness} (iii)
		we find $\lchi_t\to\chi_t$ strongly in $L^2(Q)$ as $\lambda\to 0$.
		Taking also the boundedness and continuity of $\xi'$
		(see (A3) and (B2)) into account we observe that
		$f_\lambda\weakstarlim 0$ weakly-star in $L^\infty(Q)$ as $\lambda\to 0$
		and in particular
		\begin{align}
			\int_0^t\int_\Omega f_\lambda |\dchi_t|^2\dxs \to 0\quad\text{ as }\lambda\to 0.
				\label{fLambdaConv}
		\end{align}
		Applying all the estimates for $T_1,\ldots,T_8$ we obtain
		\begin{align*}
			&\|\ly_{t}(t)\|_{L^2}^2
				+\|\nabla\lz(t)\|_{L^2}^2
				+\|\e(\ly_t)\|_{L^2(0,t;L^2)}^2
				+\|\lz_t\|_{L^2(0,t;H^1)}^2\\
				&\leq
					\delta\|\lz_t\|_{L^2(0,t;H^1)}^2
					+\delta\|\e(\ly_t)\|_{L^2(0,t;L^2)}^2
					+C_\delta\|\lz\|_{L^2(0,t;H^1)}^2
					+C_\delta\|\e(\ly)\|_{L^2(0,t;L^2)}^2\\
					&\quad
					+|\lambda|^2C_\delta\int_0^t\int_\Omega f_\lambda |\dchi_t|^2\dxs
					+C_\delta\|\lchi-\chi\|_{H^1(H^1)}^4
					+|\lambda|^2 C_\delta\|\e(\lu-u)\|_{L^2(L^2)}^2\\
					&\quad
					+|\lambda|^2 C_\delta\|\e(\lu-u)\|_{L^2(L^2)}.
		\end{align*}
		Adding $\|\lz\|_{L^2(0,t;H^1)}^2+\|\ly\|_{L^2(0,t;H^1)}^2+\|\ly_t\|_{L^2(0,t;L^2)}^2$ on both sides,
		using the estimate
		\begin{align*}
			\frac12\|\lz_t\|_{L^2(0,t;L^2)}\geq{}&c\|\lz(t)\|_{L^2}^2
		\end{align*}
		on the left-hand side and the estimate
		\begin{align*}
			\|\ly\|_{L^2(0,t;H^1)}^2\leq{}& C\int_0^t\|\ly_t\|_{L^2(0,s;H^1)}^2\ds
		\end{align*}
		on the right-hand side,
		applying Korn's inequality and choosing $\delta>0$ small, we obtain
		\begin{align*}
			&\|\ly_{t}(t)\|_{L^2}^2
				+\|\lz(t)\|_{H^1}^2
				+\|\ly\|_{H^1(0,t;H^1)}^2
				+\|\lz\|_{H^1(0,t;H^1)}^2\\
				&\quad\leq
					C\int_0^t\|\lz\|_{H^1}^2
					+\|\ly_t\|_{L^2(0,s;H^1)}^2
					+\|\ly_t\|_{L^2}^2\ds\\
					&\qquad
					+C\|\lchi-\chi\|_{H^1(H^1)}^4
					+C|\lambda|^2 \|\lu-u\|_{L^2(H^1)}^2
					+C|\lambda|^2 \|\lu-u\|_{L^2(H^1)}\\
					&\qquad
					+C|\lambda|^2\int_Q f_\lambda |\dchi_t|^2\dxs.
		\end{align*}
		Gronwall's inequality yields
		\begin{align*}
			&\|\ly_{t}(t)\|_{L^2}^2
				+\|\lz(t)\|_{H^1}^2
				+\|\ly\|_{H^1(0,t;H^1)}^2
				+\|\lz\|_{H^1(0,t;H^1)}^2\\
				&\quad\leq
					C\Big(
					\|\lchi-\chi\|_{H^1(H^1)}^4
					+|\lambda|^2 \|\lu-u\|_{L^2(H^1)}^2
					+|\lambda|^2 \|\lu-u\|_{L^2(H^1)}
					+|\lambda|^2\int_Q f_\lambda |\dchi_t|^2\dxs\Big).
		\end{align*}
		Due to Lipschitz continuity of the solution operator on bounded subsets
		of $\C B$ as in Theorem \ref{wellposedness} (iii)
		(we consider a ball containing $b$ and $b+\lambda h$), we find
		\begin{align*}
			&\|\ly_{t}(t)\|_{L^2}^2
				+\|\lz(t)\|_{H^1}^2
				+\|\ly\|_{H^1(0,t;H^1)}^2
				+\|\lz\|_{H^1(0,t;H^1)}^2\\
				&\quad\leq
					C\Big(
					|\lambda|^4\|\db\|_{H^1(H^1)}^4
					+|\lambda|^4\|\db\|_{H^1(H^1)}^2
					+|\lambda|^3\|\db\|_{H^1(H^1)}
					+|\lambda|^2\int_Q f_\lambda |\dchi_t|^2\dxs\Big).
		\end{align*}
		Taking also \eqref{fLambdaConv} into account we end up with
		\begin{align*}
			\frac{\|\ly_{t}(t)\|_{L^2}
				+\|\lz(t)\|_{H^1}
				+\|\ly\|_{H^1(0,t;H^1)}
				+\|\lz\|_{H^1(0,t;H^1)}}{|\lambda|}
				\to 0
				\qquad\text{as }\lambda\to 0.
		\end{align*}
		\ep
	\end{proof}

	%Before we prove existence of weak solution to the adjoint system \eqref{adjPDE}
	%we will state the problem in a classical notion in order to
	%highlight the main features of the proof.
	%Via integration by parts we see that $(p,q)$ is a classical solution of
	%the adjoint system \eqref{adjPDE} if	
	%
	
\subsection{Adjoint state problem}
	Let us firstly give a short motivation for the derivation of the adjoint system
	and then continue with rigorous analysis:
	
	By utilizing the differentiability of the solution operator $\C S$
	obtained in Proposition \ref{prop:strongDiff} we find for the derivative of the cost-functional $\C J$ composed
	with the $\chi$-part of the solution operator $\C S:b\mapsto(u(b),\chi(b))$ via the chain rule
	\begin{align}
		\big\langle D_b\C J(\chi(b),b),h\big\rangle
		=\big\langle \partial_\chi\C J(\chi(b),b),\dot\chi[h]\big\rangle
			+\big\langle \partial_b\C J(\chi(b),b),h\big\rangle\qquad\text{for all }h\in\C B,
		\label{chainRule}
	\end{align}
	where $\partial_\chi$ and $\partial_b$ denote the partial derivatives
	with respect to the corresponding variables.
	Now, to rewrite the expression in terms of PDEs, the adjoint system is introduced as follows: 
%	since for every direction $h$ the
%	solution $\dot\chi$ of the linearized problem needs to be computed.
	%one introduces the adjoint system.
	Our goal is to find a pair of functions $(p,q)\in \aQ$ such that
	\begin{align}
		\big\langle \partial_\chi\C J(\chi(b),b),\dot\chi[h]\big\rangle
			=\big\langle (p,q),\C C(h)\big\rangle,
		\label{pqId}
	\end{align}
	where $\C C:\C B\to\aQ^*$ (note that $\aQ^*$ denotes the topological dual of $\aQ$)
	specifies the operator mapping the control variable
	$b$ to the right-hand side of \eqref{linUPDE}-\eqref{linChiPDE}.
	More precisely
	\begin{align}
		\big\langle (p,q),\C C(h)\big\rangle=\int_\Sigma p\cdot h\dxt,
		\label{pqId2}
	\end{align}
	i.e. the right-hand side of \eqref{linUPDE}-\eqref{linChiPDE} tested with $(p,q)$.
	We call $(p,q)$ the adjoint variables to the linearized solutions $(\du,\dchi)$
	at $(u,\chi)$. Even though the adjoint variable $q$ does not appear
	directly in \eqref{pqId} (by taking \eqref{pqId2} into account) it will be used for intermediate
	steps as shown below.
	
	In order to derive an explicit PDE system for $(p,q)$ (which will be justified rigorously afterward),
	we proceed formally and test \eqref{linUPDE}-\eqref{linChiPDE} with $(p,q)$.
	Then, adding both resulting equations yield
	\begin{align}
	\begin{split}
		&\int_0^T\langle \du_{tt},p\rangle_{H^1}\dt+\int_Q \mathbb C'(\chi)\dchi\e(u):\e(p)+\mathbb C(\chi)\e(\du):\e(p)+\mathbb D\e(\du_t):\e(p)\dxt\\
		&+\int_Q \nabla\dchi\cdot\nabla\psi+\nabla\dchi_t\cdot\nabla\psi+\dchi_t\psi+\xi'(\chi_t)\dchi_t\psi+\frac{1}{2}\mathbb C''(\chi)\dchi\e(u):\e(u)\psi\dxt\\
		&+\int_Q \mathbb C'(\chi)\e(\du):\e(u)\psi+f''(\chi)\dchi\psi\dxt\\
		&\quad=\big\langle (p,q),\C C(h)\big\rangle.
	\end{split}
	\label{linPQ}
	\end{align}
	Consequently the adjoint variables should satisfy for all ``appropriate'' test-functions $(\ph,\psi)$:
	\begin{align}
	\begin{split}
		&\int_0^T\langle \ph_{tt},p\rangle_{H^1}\dt+\int_Q \mathbb C'(\chi)\psi\e(u):\e(p)+\mathbb C(\chi)\e(\ph):\e(p)+\mathbb D\e(\ph_t):\e(p)\dxt\\
		&+\int_Q \nabla\psi\cdot\nabla q+\nabla\psi_t\cdot\nabla q+\psi_t q+\xi'(\chi_t)\psi_t q+\frac{1}{2}\mathbb C''(\chi)\psi\e(u):\e(u)q\dxt\\
		&+\int_Q \mathbb C'(\chi)\e(\ph):\e(u)q+f''(\chi)\psi q\dxt\\
		&\quad=\big\langle \partial_\chi\C J(\chi(b),b),\psi\big\rangle.
	\end{split}
	\label{adjointSum}
	\end{align}
	In this case we can recover \eqref{pqId} by using \eqref{adjointSum}
	tested with $(\ph,\psi)=(\du,\dchi)$ and using \eqref{linPQ}.
	Note that \eqref{adjointSum} can be equivalently recasted as the system
	\begin{align}
		&\int_0^T\langle p,\varphi_{tt}\rangle_{H^1}\dt+\int_{Q}\mathbb C(\chi)\e(p):\e(\varphi)+\mathbb C'(\chi)\e(u)q:\e(\varphi)+\mathbb D\e(p):\e(\varphi_t)\dxt=0,
		\label{adjP}\\
		&
		\begin{aligned}
			&\int_{Q}q\psi_t+\xi'(\chi_t)q\psi_t
				+\nabla q\cdot\nabla\psi+\nabla q\cdot\nabla\psi_t
				+\frac 12 \mathbb C''(\chi)\e(u):\e(u)q\psi\dxt\\
			&\quad+\int_{Q}\mathbb C'(\chi)\e(u):\e(p)\psi+f''(\chi)q\psi\dxt=\int_\Omega\lambda_T\big(\chi(T)-\chi_T\big)\psi(T)\dx.
		\end{aligned}
		\label{adjQ}
	\end{align}
	%In this Section we discuss a system which is formally the adjoint system to our optimal
	%control problem \textbf{(CP)}. We derive this adjoint system by utilizing the formal
	%lagrange method discussed intesively in the book of Troeltzsch. To this end, assume that
	%$b\in \C B_{adm}$ is a minimizer of \textbf{(CP)} with the corresponding state variable
	%$(u,\chi)=\C S(b)$.
	In the pointwise formulation this reads as follows:
	\begin{align*}
		&p_{tt}-\DIV(\mathbb C(\chi)\e(p)+\mathbb C'(\chi)\e(u)q-\mathbb D\e(p_t))=0
			&&\text{in }Q,\\
		&-q_t-(\bxi'(\chi_t)q)_t+\Delta q_t-\Delta q
			+\frac12\mathbb C''(\chi)\e(u):\e(u)q+\mathbb C'(\chi)\e(u):\e(p)+f''(\chi)q=0
			&&\text{in }Q,\\
		&(\mathbb C(\chi)\e(p)+\mathbb C'(\chi)\e(u)q-\mathbb D\e(p_t))\cdot\nu=0
			&&\text{on }\Sigma,\\
		&\nabla p\cdot\nu=0
			&&\text{on }\Sigma\\
		&\hspace*{-0.0em}\text{with the final-time conditions}\\
			\hspace*{0.0em}&p(T)=p_t(T)=0
			&&\text{in }\Omega,\\
		&-\Delta q(T)+q(T)+\bxi'(\chi_t(T))q(T)=\lambda_T(\chi(T)-\chi_T)\hspace*{10.6em}
			&&\text{in }\Omega,\\
		&\nabla q(T)\cdot\nu=0
			&&\text{on }\Gamma.\;\;
	\end{align*}
	The PDE system above is a backward in time boundary value problem for $(p,q)$,
	where $q$ itself fulfills an elliptic PDE at the final-time $T$.
%		As seen in the motivation above
%		the above system is the adjoint problem to the optimal control problem \textbf{(CP)}. 
%		This will be shown rigorously in the next section.
		Our task is now to prove existence of solutions in a weak sense.
%		This is one of the main tasks of this paper. 
	\begin{proposition}[Existence of very weak solution to the adjoint problem]
		\label{prop:adjointSystem}
		Suppose that the assumptions \textbf{(A1)}-\textbf{(A5)}, \textbf{(O1)}-\textbf{(O2)}
		and \textbf{(B1)}-\textbf{(B3)} are fulfilled.
		Furthermore let $(u,\chi)$ be a solution of the state system
		to $b\in\C B$, i.e. $(u,\chi):=\C S(b)$.
		Then there exists a pair of function $(p,q)\in\aQ$ (weak solution) such that
		\begin{align*}
			\text{\eqref{adjP}-\eqref{adjQ} holds for all test-functions }
			(\varphi,\psi)\in \dQ
			\text{ with }\varphi(0)=\varphi_t(0)=\psi(0)=0.
		\end{align*}
	\end{proposition}
	
	We prove Proposition \ref{prop:adjointSystem} in several steps
	whose intermediate results are highlighted by corresponding lemmas.
	The idea is the following:
	
	First, we will apply a time transformation and work with
	a regularized state variable $\chi_\alpha$ instead of $\chi$ with index $\alpha>0$.
	This will enable us to time-differentiate the nonlinear term $\bxi'(\chi_t)q$
	and obtain suitable a priori estimates.
	Existence of solutions for this regularized system
	will be achieved by utilizing a time-discretization scheme with time step size $\tau$
	and a limit analysis $\tau\downarrow 0$.
	In the time-discrete setting the regularized adjoint system is an elliptic problem
	which can be solved by standard methods.
	We then derive a priori estimates (energy estimates) uniformly in $\tau$ in order to pass to the limit $\tau\downarrow0$.
	After solving the regularized adjoint system, i.e. for $\alpha>0$,
	we transform it to the very weak formulation
	as used in \eqref{adjP}-\eqref{adjQ}, where time-derivatives only occur on the test-functions.
	Then, roughly speaking, we test the resulting system with certain modified anti-derivatives
	with respect to time
	of $p_\alpha$ and $q_\alpha$ and end up with
	a priori estimates uniformly in $\alpha$ in the large space $\aQ$.
	Finally the limit passage $\alpha\downarrow0$ can be performed in the regularized adjoint system.
	\vspace*{0.5em}\\\textit{\underline{Step 1:} setup time-transformation and $\alpha$-regularization}\vspace*{0.5em}\\
			In the first step we consider a transformation of the adjoint system above
			to an inital-boundary value problem
			by using the time transformation $t\mapsto T-t$.
			We find
			(we keep the notation $(p,q)$ for the transformed variables)
			\begin{align}
				&p_{tt}-\DIV(\mathbb C(\chi)\e(p)+\mathbb C'(\chi)\e(u)q+\mathbb D\e(p_t))=0
					&&\text{in }Q,\label{ADJmech}\\
				&q_t+(\bxi'(-\chi_t)q)_t-\Delta q_t-\Delta q
					+\frac12\mathbb C''(\chi)\e(u):\e(u)q\notag\\
				&\hspace*{12.3em}+\mathbb C'(\chi)\e(u):\e(p)+f''(\chi)q=0
					&&\text{in }Q,\label{ADJdamage}\\
				&(\mathbb C(\chi)\e(p)+\mathbb C'(\chi)\e(u)q+\mathbb D\e(p_t))\cdot\nu=0
					&&\text{on }\Sigma,\label{ADJbc1}\\
				&\nabla p\cdot\nu=0
					&&\text{on }\Sigma\label{ADJbc2}
			\end{align}
			with the initial conditions
			\begin{align}
				\hspace*{2.3em}&p(0)=p_t(0)=0
					&&\text{in }\Omega,\label{ADJic1}\\
				&-\Delta q(0)+q(0)+\bxi'(-\chi_t(0))q(0)=\lambda_T(\chi(T)-\chi_T)\hspace*{6.4em}
					&&\text{in }\Omega,\label{ADJic2}\\
				&\nabla q(0)\cdot\nu=0
					&&\text{on }\Gamma.\label{ADJic3}\;\;
			\end{align}
			Secondly, to obtain rigorous existence results,
			we will firstly work with a regularized version of the state variable $(u,\chi)$.
			To this end, let
			$\{\chi_\alpha\}\subseteq C^\infty(\ol Q)$ be a smooth approximation sequence such that
			$$
				\chi_\alpha\to \chi\text{ in } H^1(0,T;H^2(\Omega))\text{ as }\alpha\downarrow 0.
			$$
			The regularized system is obtained by replacing $\chi$ by its regularization $\chi_\alpha$
			in (\ref{ADJmech})-(\ref{ADJic3}).
			In the next two steps of the proof we will prove existence of solutions via a time-discretization
			argument. In the last step we will perform $\alpha\downarrow 0$.
		\vspace*{0.5em}\\\textit{\underline{Step 2:} setup time-discretization of the $\alpha$-regularized problem}\vspace*{0.5em}\\
			To keep the notation simple we omit the explicit dependence on $\alpha>0$ in this step. We consider the following time-discretization scheme:
			Let $\{0,\tau,2\tau,\ldots,T\}$ denote an equidistant partition of $[0,T]$
			with time step size $\tau:=T/M$ and $M\in\N$.
			Moreover, let denote the first and second difference operators by 
			$\Dt(p):=\frac{p^k-p^{k-1}}{\tau}$ and $\Dt(\Dt(p)):=\frac{p^k-2p^{k-1}+p^{k-2}}{\tau^2}$. For an arbitrary sequence $\{h^k\}_{k=0,\ldots,M}$ we define
			the piecewise constant and linear interpolation as
			\begin{align}\label{eqn:interpolation}
				&\displaystyle\ol{h}_\tau(t):=h^k,\, \ul{h}_\tau(t):=h^{k-1},\, \displaystyle h_\tau(t):=\frac{t-(k-1)\tau}{\tau}h^k+\frac{k\tau-t}{\tau}h^{k-1}\quad\text{for $t\in((k-1)\tau,k\tau]$.}
			\end{align}
			With these preparations the time-discretized version of the system in step 1 reads
			in a weak formulation as 
			\begin{align}
				&\text{find $\{(p^k,q^k)\}_{k=1,\ldots,M}\subseteq H^1(\Omega;\Rn)\times H^1(\Omega)$ such that}\notag\\
				&\int_\Omega \Dt(\Dt(p))\cdot\ph+\big(\CC(\chi^k)\e(\pk)+\CC'(\chi^k)\e(\uk)\qkk+\DD\e(\Dt(p))\big):\e(\ph)\dx=0
					\label{eqn:adjDiscrEq1}\\
				&\int_\Omega \Dt(q)\psi+\sA^k\Dt(q)\psi-\sB^k\qk\psi
					+\Dt(\nabla \qk)\cdot\nabla\psi+\nabla\qk\cdot\nabla\psi\notag\\
				&\qquad+\frac 12 \CC''(\chik)\e(\uk):\e(\uk)\qk\psi+\CC'(\chik)\e(\uk):\e(\pk)\psi+f''(\chik)\qk\psi\dx=0,
					\label{eqn:adjDiscrEq2}
			\end{align}
			for all $(\ph,\psi)\in H^1(\Omega;\Rn)\times H^1(\Omega)$,
			where $\{\sA^k\}_{k=0,\ldots,M}$ with $\sA^k\geq0$ and $\{\sB^k\}_{k=0,\ldots,M}$ are time-discretizations
			of $\xi'(-\partial_t\chi_\alpha)$ and
			$\xi''_\beta(-\partial_t\chi_\alpha)\partial_{tt}\chi_\alpha$
			respectively, such that
			\begin{align*}
				&\ol\sA\to \xi'(-\partial_t\chi_\alpha)
					&&\text{strongly in }L^\infty(Q),\\
				&\ol\sB\to \xi''_\beta(-\partial_t\chi_\alpha)\partial_{tt}\chi_\alpha
					&&\text{strongly in }L^\infty(Q)
			\end{align*}
			\begin{remark}
				Note that in (\ref{eqn:adjDiscrEq1})-(\ref{eqn:adjDiscrEq2}) we use a time-discretization
				in a form which allows for a decoupled system of linear elliptic equations. 
			\end{remark}
			We proceed recursively and construct $\pk,\qk$ from $\pkk,\pkkk,\qkk$.
			The initial values are given by
			$p^0=p^{-1}=0$ and $q^0$ is the weak solution of
			\begin{align}
				\int_\Omega \nabla q^0\cdot\nabla\ph+q^0\ph+\bxi'(-\partial_t\chi_\alpha(0))q^0\ph -\lambda_T\big(\chi(T)-\chi_T\big)\ph\dx=0
					\qquad\text{for all }\ph\in H^1(\Omega).
					\label{initialQ}
			\end{align}
			\begin{lemma}
			There exists $\{(p^k,q^k)\}_{k=1,\ldots,M}\subseteq H^1(\Omega;\Rn)\times H^1(\Omega)$, which fulfill (\ref{eqn:adjDiscrEq1})-(\ref{eqn:adjDiscrEq2}).
			\end{lemma}
			\begin{proof}
			Employing Lax-Milgram's theorem for given $\qkk\in H^1(\Omega)$ the equation (\ref{eqn:adjDiscrEq1}) admits a solution $p^k\in H^1(\Omega;\Rn)$. Moreover, by standard theory of partial differential equations 	of second order we obtain for given $p^k\in H^1(\Omega;\Rn)$ a solution $q^k\in H^1(\Omega)$ to the equation (\ref{eqn:adjDiscrEq2}). 
			\end{proof}
			%\vspace*{0.5em}\\
		\textit{\underline{Step 3:} a priori estimates for the $\alpha$-regularized time-discrete system and limit passage $\tau\downarrow0$}\vspace*{0.5em}\\
			After setting up the time-discrete scheme and existence, we will establish
			a priori estimates uniformly in $\tau$ in order to perform $\tau\downarrow 0$.
			\begin{lemma}
			There exists a constant $C>0$ (possibly depending on $\alpha$) independent of $\tau$ such that
			\begin{align*}
				\|(p_\tau,q_\tau)\|_{\dot{\C Q}}\leq C.
			\end{align*}
			\end{lemma}
			\begin{proof}
				By testing \eqref{eqn:adjDiscrEq1} with $\pk-\pkk$, testing \eqref{eqn:adjDiscrEq2} with
				$\qk-\qkk$ and summing over $k=1,\ldots,\ol t/\tau$, we obtain the estimates
				\begin{align}
					&\begin{aligned}
						&\frac 12\|\partial_t p(t)\|_{L^2}^2+c\|\e(\partial_t p)\|_{L^2(0,\olt;L^2)}^2\\
						&\qquad\leq \underbrace{-\int_0^\olt\int_\Omega\CC(\ol\chi)\e(\ol p):\e(\partial_t p)\dxs}_{=:T_1}
						\underbrace{-\int_0^\olt\int_\Omega\CC'(\ol\chi)\e(\ol u):\e(\partial_t p) \ul{q}\dxs}_{=:T_2}
					\end{aligned}
					\label{adjEst1befor1}\\
					&\begin{aligned}
						&\frac 12\|\nabla\ol q(t)\|_{L^2}^2+\|\partial_t q\|_{L^2(0,\olt;H^1)}^2\\
						&\qquad\leq \underbrace{-\int_0^\olt\int_\Omega\ol\sA|\partial_t q|^2\dxs}_{=:T_3}
						\underbrace{-\int_0^\olt\int_\Omega\ol\sB\ol q\,\partial_t q\dxs}_{=:T_4}
						\underbrace{-\int_0^\olt\int_\Omega\frac 12 \CC''(\ol\chi)\e(\ol u):\e(\ol u)\ol q\,\partial_t q\dxs}_{=:T_5}\\
						&\qquad\quad
						\underbrace{-\int_0^\olt\int_\Omega\CC'(\ol\chi)\e(\ol u):\e(\ol p)\,\partial_t q\dxs}_{=:T_6}
						\underbrace{-\int_0^\olt\int_\Omega f''(\ol\chi)\ol q\,\partial_t q\dxs}_{=:T_7}
					\end{aligned}
					\label{adjEst1befor2}
				\end{align}
				H\"older's and Young's inequality show
				\begin{align*}
					T_1\leq{}&\|\CC(\ol\chi)\|_{L^\infty(L^\infty)}\|\e(\ol p)\|_{L^2(0,\olt;L^2)}\|\e(\partial_t p)\|_{L^2(0,\olt;L^2)}\\
						&\delta\|\e(\partial_t p)\|_{L^2(0,\olt;L^2)}^2+C_\delta\|\e(\ol p)\|_{L^2(0,\olt;L^2)}^2,\\
					T_2\leq{}&\|\CC'(\ol\chi)\|_{L^\infty(L^\infty)}\|\e(\ol u)\|_{L^\infty(L^4)}\|\e(\partial_t p)\|_{L^2(0,\olt;L^2)}\|\ul q\|_{L^2(0,\olt;L^4)}\\
						&\delta\|\e(\partial_t p)\|^2_{L^2(0,\olt;L^2)}+C_\delta\|\ul q\|^2_{L^2(0,\olt;H^1)},\\
					T_3\leq{}&0,\\
					T_4\leq{}&\|\ol\sB\|_{L^\infty(L^\infty)}\|\ol q\|_{L^2(0,\olt;L^2)}\|\partial_t q\|_{L^2(0,\olt;L^2)}\\
						&\delta\|\partial_t q\|^2_{L^2(0,\olt;L^2)}+C_\delta\|\ol q\|^2_{L^2(0,\olt;L^2)},\\
					T_5\leq{}&\|\CC''(\ol\chi)\|_{L^\infty(L^\infty)}\|\e(\ol u)\|_{L^\infty(L^6)}^2\|\ol q\|_{L^2(0,\olt;L^6)}\|\partial_t q\|_{L^2(0,\olt;L^2)}\\
						&\delta\|\partial_t q\|^2_{L^2(0,\olt;L^2)}+C_\delta\|\ol q\|^2_{L^2(0,\olt;H^1)},\\
					T_6\leq{}&\|\CC'(\ol\chi)\|_{L^\infty(L^\infty)}\|\e(\ol u)\|_{L^\infty(L^4)}\|\e(\ol p)\|_{L^2(0,\olt;L^2)}\|\partial_t q\|_{L^2(0,\olt;L^4)}\\
						&\leq\delta\|\partial_t q\|^2_{L^2(0,\olt;H^1)}+C_\delta\|\e(\ol p)\|^2_{L^2(0,\olt;L^2)},\\
					T_7\leq{}&\|f''(\ol\chi)\|_{L^\infty(L^\infty)}\|\ol q\|_{L^2(0,\olt;L^2)}\|\partial_t q\|_{L^2(0,\olt;L^2)}\\
						&\delta\|\partial_t q\|^2_{L^2(0,\olt;L^2)}+C_\delta\|\ol q\|^2_{L^2(0,\olt;L^2)}.
				\end{align*}
				All in all we obtain by adding the inequalities in \eqref{adjEst1befor1}-\eqref{adjEst1befor2}, applying above estimates and readjusting the constants $(\delta,C_\delta)$
				\begin{align*}
						&\|\partial_t p(t)\|_{L^2}^2+\|\e(\partial_t p)\|_{L^2(0,\olt;L^2)}^2
							+\|\nabla\ol q(t)\|_{L^2}^2+\|\partial_t q\|_{L^2(0,\olt;H^1)}^2\\
						&\qquad\leq
							\delta\big(\|\e(\partial_t p)\|^2_{L^2(0,\olt;L^2)}+\|\partial_t q\|^2_{L^2(0,\olt;H^1)}\big)
							+C_\delta\big(\|\e(\ol p)\|^2_{L^2(0,\olt;L^2)}+\|\ul q\|^2_{L^2(0,\olt;H^1)}\big).
				\end{align*}
				Furthermore, observe that
				\begin{align*}
				&\|\ul q\|^2_{L^2(0,\olt;H^1)}\leq C\left(\|q(0)\|^2_{H^1}+\int_0^\olt\|\partial_t q\|^2_{L^2(0,s;H^1)}ds\right),\\
				&\|\e(\ol p)\|^2_{L^2(0,\olt;L^2)}\leq C\left(\|\e(\ol p(0))\|^2_{L^2}+\int_0^\olt\|\e(\ol p)\|^2_{L^2(0,s;L^2)}ds\right).
				\end{align*}
				By means of Korn's inequality and Gronwall's lemma we find the a priori estimates
				\begin{align*}
					\|p_\tau\|_{W^{1,\infty}(0,T;L^2)\cap H^1(0,T;H^1)}
							+\|q_\tau\|_{H^1(0,T;H^1)}\leq C.
				\end{align*}
				By comparison in \eqref{eqn:adjDiscrEq1} we also get
				\begin{align*}
					\|p_\tau\|_{H^2(0,T;(H^1)^*)}\leq C.
				\end{align*}
				\ep
			\end{proof}
			Now, extracting weakly convergent subsequences we may pass to the limit as $\tau\downarrow 0$
			in \eqref{eqn:adjDiscrEq1}-\eqref{eqn:adjDiscrEq2} and obtain $(p,q)\in \dQ$ fulfilling 
			\begin{align}
				&\langle \partial_{tt}p,\ph\rangle_{H^1}+\int_\Omega\big(\CC(\chi)\e(p)+\CC'(\chi)\e(u)q+\DD\e(\partial_t p)\big):\e(\ph)\dx=0
					\label{adjEq1}\\
				&\begin{aligned}
					&\int_\Omega (\partial_t q)\psi+\xi'(-\partial_t\chi_\alpha)(\partial_t q)\psi-\xi''_\beta(-\partial_t\chi_\alpha)(\partial_{tt}\chi_\alpha) q\psi
					+\nabla \partial_t q\cdot\nabla\psi+\nabla q\cdot\nabla\psi\\
					&\qquad+\frac 12 \CC''(\chi)\e(u):\e(u)q\psi+\CC'(\chi)\e(u):\e(p)\psi+f''(\chi)q\psi\dx=0
				\end{aligned}
					\label{adjEq2}
			\end{align}
			for a.e. $t\in(0,T)$, all $(\ph,\psi)\in H^1(\Omega;\Rn)\times H^1(\Omega)$ and with the initial conditions $p(0)=p_t(0)=0$ and $q(0)=q^0$ satisfying
			\eqref{initialQ}.
			\vspace*{0.5em}\\\hspace*{-0.27em}\textit{\underline{Step 4:} a priori estimates for the $\alpha$-regularized time-continuous system and limit passage $\alpha\downarrow0$}\vspace*{0.5em}\\
			In this step we are going to derive certain weak a priori estimates uniformly in $\alpha$.
			In preparation of the corresponding result we prove a technical Lemma. 
			\begin{lemma}
			\label{lemma:parInt}
			Let $s,t\in[0,T]$ be given and define
			\begin{align}
				&\hat p_\alpha^t(s):=
				\begin{cases}
					0&\text{if }s\in[t,T],\\
					\int_{s}^{t} p_\alpha(\tau)\mathrm d\tau&\text{if } s\in[0,t),
				\end{cases}
				&&\hat q_\alpha^t(s):=
				\begin{cases}
					0&\text{if }s\in[t,T],\\
					\int_{s}^{t} q_\alpha(\tau)\mathrm d\tau&\text{if } s\in[0,t).
				\end{cases}
			\label{timeIntSol}
			\end{align}
			Then, it holds
				\begin{align*}
					\frac{\mathrm d}{\mathrm dt}\Big(\|\hat q_\alpha^t(0)\|_{H^1}^2\Big)
						=2\int_\Omega\Big(\hat q_\alpha^t(0) q_\alpha(t)+\nabla\hat q_\alpha^t(0)\cdot\nabla q_\alpha(t)\Big)\dx.
				\end{align*}
				Similarly for $\hat p_\alpha^t$.
			\end{lemma}
			\begin{proof}
						In order to differentiate the parametrized integral 
						$$
							t\mapsto\int_\Omega f(x,t)\dx\;\text{ with }\;
							f(x,t):=|\hat q_\alpha^t(x,0)|^2+|\nabla\hat q_\alpha^t(x,0)|^2
						$$
						we apply \cite[5.7 Satz -- Zusatz (Differentiation unter dem Integralzeichen)]{El09} and
						check the following properties by noticing that $q_\alpha\in H^1(0,T;H^1(\Omega))$:
						\begin{itemize}
							\item
								For every $t\in[0,T]$ the function $f(\cdot,t)$ is in $L^1(\Omega)$.
							\item
								For a.e. $x\in\Omega$ and $t\in(0,T)$
								the function $\partial_t f(x,t)$ is differentiable with respect to $t$
								and for the derivative we obtain
								\begin{align*}
									\partial_s f(x,t)
									=2\hat q_\alpha^t(x,0)q_\alpha(x,t)+2\nabla \hat q_\alpha^t(x,0)\cdot \nabla q_\alpha(x,t).
								\end{align*}
							\item
								Boundedness of the partial derivative:
								\begin{align*}
									|\partial_t f(x,t)|
										\leq{}& 2\|q_\alpha(x,\cdot)\|_{L^1(0,T)} \|q_\alpha(x,\cdot)\|_{L^\infty(0,T)}
											+2\|\nabla q_\alpha(x,\cdot)\|_{L^1(0,T)} \|\nabla q_\alpha(x,\cdot)\|_{L^\infty(0,T)}\\
										\leq{}& C\|q_\alpha(x,\cdot)\|_{L^1(0,T)} \|q_\alpha(x,\cdot)\|_{H^1(0,T)}
											+C\|\nabla q_\alpha(x,\cdot)\|_{L^1(0,T)} \|\nabla q_\alpha(x,\cdot)\|_{H^1(0,T)}.
								\end{align*}
%								Here we used the identity
%								$\nabla \hat q_\alpha^t(0) = \int_0^t\nabla q_\alpha\ds$
%								which follows from definition of the weak derivative. 
						\end{itemize}
				\ep
			\end{proof}
			
			\begin{lemma}
			\label{lemmma:adjPQ}
				There exists a constant $C>0$ independent of $\alpha$ such that
				\begin{align*}
					\|(p_\alpha,q_\alpha)\|_{\aQ}\leq C.
				\end{align*}
			\end{lemma}
			\begin{proof}
			To this end, let $(p_\alpha,q_\alpha)\in \dQ$ be a solution of \eqref{adjEq1}-\eqref{adjEq2}
			for $\alpha>0$ as proven in step 2.
			Integrating \eqref{adjEq1}-\eqref{adjEq2} in time, applying integration by part and using the initial conditions yield
			\begin{align}
				&\int_0^T\int_{\Omega}-\partial_t p_{\alpha}\cdot\partial_t\ph+\big(\CC(\chi_{\alpha})\e(p_\alpha)+\CC'(\chi_{\alpha})\e(u_{\alpha})q_\alpha+\DD\e(\partial_t p_\alpha)\big):\e(\ph)\dxt=0,
					\label{eqn:adjContEq1}\\
				&\begin{aligned}
					&\int_0^T\int_{\Omega}-q_\alpha\partial_t\psi-\xi'(-\partial_t\chi_\alpha)q_\alpha\partial_t\psi
						-\nabla q_\alpha\cdot\nabla\partial_t\psi+\nabla q_\alpha\cdot\nabla\psi\dxt\\
						&\quad+\int_0^T\int_{\Omega}\frac 12 \CC''(\chi_{\alpha})\e(u_{\alpha}):\e(u_{\alpha})q_\alpha\psi+\CC'(\chi)\e(u):\e(p_\alpha)\psi+f''(\chi)q_\alpha\psi\dxt\\
						&\qquad=\int_\Omega\lambda_T\big(\chi_{\alpha}(T)-\chi_T\big)\psi(0)\dx
				\end{aligned}
				\label{eqn:adjContEq2}
			\end{align}
			for all $(\ph,\psi)\in\dQ$ with $\ph(T)=0$ and $\psi(T)=0$.
			
			Testing \eqref{eqn:adjContEq1} with $\hat p_\alpha^t$
			and \eqref{eqn:adjContEq2} with $\hat q_\alpha^t$
			and noticing $p_\alpha=-\partial_t\hat p_\alpha^t$ and $q_\alpha=-\partial_t\hat q_\alpha^t$
			as well as the initial and final-time conditions $p_\alpha(0)=0$ and $\hat p_\alpha^t(T)=0$,
			we obtain after integration by parts 
			\begin{align*}
				&\int_0^t\int_{\Omega}\partial_t p_\alpha\cdot p_\alpha+\CC(\chi_{\alpha})\e(p_\alpha):\e(\hat p_\alpha^t)+\CC'(\chi_{\alpha})\e(u_{\alpha})q_\alpha:\e(\hat p_\alpha^t)+\DD\e(p_\alpha):\e(p_\alpha)\dxs=0,\\
				&\int_0^t\int_{\Omega}|q_\alpha|^2+\xi'(-\partial_t\chi_\alpha)|q_\alpha|^2+|\nabla q_\alpha|^2+\nabla(-\partial_t\hat q_\alpha^t)\cdot\nabla\hat q_\alpha^t
					+\frac 12 \CC''(\chi_{\alpha})\e(u_{\alpha}):\e(u_{\alpha})q_\alpha\hat q_\alpha^t\dxs\\
				&\quad+\int_0^t\int_{\Omega}\CC'(\chi_{\alpha})\e(u_{\alpha}):\e(p_\alpha)\hat q_\alpha^t+f''(\chi_{\alpha})q_\alpha\hat q_\alpha^t\dxs
					=\int_\Omega\lambda_T\big(\chi_{\alpha}(T)-\chi_T\big)\hat q_\alpha^t(0)\dx.
			\end{align*}
			Adding these equations, using $\xi'\geq 0$ (see \textbf{(A3)}) and $q_\alpha^t(T)=0$ and applying further standard estimates yield
			\begin{align}
			\begin{aligned}
				&\|p_\alpha(t)\|_{L^2}^2+c\|\e(p_\alpha)\|_{L^2(0,t;L^2)}^2+\|q_\alpha\|_{L^2(0,t;H^1)}^2+\|\nabla\hat q_\alpha^t(0)\|_{L^2}^2\\
					&\qquad\leq
						\underbrace{-\int_0^t\int_{\Omega}\CC(\chi_{\alpha})\e(p_\alpha):\e(\hat p_\alpha^t)\dxs}_{=:T_1}
						\underbrace{-\int_0^t\int_{\Omega}\CC'(\chi_{\alpha})\e(u_{\alpha})q_\alpha:\e(\hat p_\alpha^t)\dxs}_{=:T_2}\\
					&\qquad\quad
						\underbrace{-\int_0^t\int_{\Omega}\frac 12 \CC''(\chi_{\alpha})\e(u_{\alpha}):\e(u_{\alpha})q_\alpha\hat q_\alpha^t\dxs}_{=:T_3}
						\underbrace{-\int_0^t\int_{\Omega}\CC'(\chi_{\alpha})\e(u_{\alpha}):\e(p_\alpha)\hat q_\alpha^t\dxs}_{=:T_4}\\
					&\qquad\quad
						\underbrace{-\int_0^t\int_{\Omega}f''(\chi_{\alpha})q_\alpha\hat q_\alpha^t\dxs}_{=:T_5}
						+\underbrace{\int_\Omega\lambda_T\big(\chi_{\alpha}(T)-\chi_T\big)\hat q_\alpha^t(0)\dx}_{=:T_6}.
			\end{aligned}
			\label{pqEst}
			\end{align}
			We obtain by standard calculations
			\begin{align*}
				T_1\leq{}&\|\CC(\chi_\alpha)\|_{L^\infty(L^\infty)}\|\e(p_\alpha)\|_{L^2(0,t;L^2)}\|\e(\hat p_\alpha^t)\|_{L^2(0,t;L^2)}\\
					\leq{}&\delta\|\e(p_\alpha)\|_{L^2(0,t;L^2)}+C_\delta\|\e(\hat p_\alpha^t)\|_{L^2(0,t;L^2)},\\
				T_2\leq{}&\|\CC'(\chi_\alpha)\|_{L^\infty(L^\infty)}\|\e(u_{\alpha})\|_{L^\infty(L^4)}\|q_\alpha\|_{L^2(0,t;L^4)}\|\e(\hat p_\alpha^t)\|_{L^2(0,t;L^2)}\\
					\leq{}&\delta\|q_\alpha\|_{L^2(0,t;H^1)}^2+C_\delta\|\e(\hat p_\alpha^t)\|_{L^2(0,t;L^2)}^2,\\
				T_3\leq{}&\frac12\|\CC''(\chi_\alpha)\|_{L^\infty(L^\infty)}\|\e(u_{\alpha})\|_{L^\infty(L^4)}^2\|q_\alpha\|_{L^2(0,t;L^4)}\|\hat q_\alpha^t\|_{L^2(0,t;L^4)}\\
					\leq{}&\delta\|q_\alpha\|_{L^2(0,t;H^1)}^2+C_\delta\|\hat q_\alpha^t\|_{L^2(0,t;H^1)}^2,\\
				T_4\leq{}&\|\CC'(\chi_\alpha)\|_{L^\infty(L^\infty)}\|\e(u_{\alpha})\|_{L^\infty(L^4)}\|\e(p_\alpha)\|_{L^2(0,t;L^2)}\|\hat q_\alpha^t\|_{L^2(0,t;L^4)}\\
					\leq{}&\delta\|\e(p_\alpha)\|_{L^2(0,t;L^2)}^2+C_\delta\|\hat q_\alpha^t\|_{L^2(0,t;H^1)}^2,\\
				T_5\leq{}&\|f''(\chi_\alpha)\|_{L^\infty(L^\infty)}\|q_\alpha\|_{L^2(0,t;L^2)}\|\hat q_\alpha^t\|_{L^2(0,t;L^2)}\\
					\leq{}&\delta\|q_\alpha\|_{L^2(0,t;L^2)}^2+C_\delta\|\hat q_\alpha^t\|_{L^2(0,t;L^2)}^2,\\
				T_6\leq{}&\lambda_T\|\chi_\alpha(T)-\chi_T\|_{L^2}\|\hat q_\alpha^t(0)\|_{L^2},\\
					\leq{}&\delta\|q_\alpha\|_{L^2(0,t;L^2)}^2+C_\delta\|\chi_\alpha(T)-\chi_T\|_{L^2}^2.
			\end{align*}
			We observe that it will be indispensable to absorb the terms
			$\|\hat q_\alpha^t\|_{L^2(0,t;H^1)}^2$ and
			$\|\e(\hat p_\alpha^t)\|_{L^2(0,t;L^2)}^2$
			by terms on the left-hand side in \eqref{pqEst}.
			To this end, we notice that by definition of $\hat q_\alpha^t$ we have
			\begin{align}
				\begin{aligned}
				\|\hat q_\alpha^t\|_{L^2(0,t;H^1)}^2
					={}&\|\hat q_\alpha^t(0)-\int_0^s q_\alpha(\tau)\mathrm d\tau\|_{L^2(0,t;H^1)}^2\\
					\leq{}&C\|\hat q_\alpha^t(0)\|_{L^2(0,t;H^1)}^2+C\int_0^t\|q_\alpha\|_{L^2(0,s;H^1)}^2.
				\end{aligned}
				\label{adjEst1}
			\end{align}
			The first term on the right-hand side of \eqref{adjEst1} is treated by a tricky calculations using Lemma \ref{lemma:parInt}:
			\begin{align*}
				&\|\hat q_\alpha^t(0)\|_{L^2(0,t;H^1)}^2\\
					&=\int_0^t\|\hat q_\alpha^t(0)\|_{H^1}^2\ds
						=t\|\hat q_\alpha^t(0)\|_{H^1}^2
						=\int_0^t \frac{\mathrm d}{\mathrm ds}\Big(s\|\hat q_\alpha^s(0)\|_{H^1}^2\Big)\ds\\
					&=\int_0^t \|\hat q_\alpha^s(0)\|_{H^1}^2\ds+\int_0^t s\frac{\mathrm d}{\mathrm ds}\Big(\|\hat q_\alpha^s(0)\|_{H^1}^2\Big)\ds\\
					&=\int_0^t \|\hat q_\alpha^s(0)\|_{H^1}^2\ds+\int_0^t 2s\int_\Omega\Big(\hat q_\alpha^s(0) q_\alpha(s)+\nabla\hat q_\alpha^s(0)\cdot\nabla q_\alpha(s)\dx\Big)\ds\\
					&\leq\int_0^t\|q_\alpha\|_{L^2(0,s;H^1)}^2\ds+2T\int_0^t\int_\Omega C_\delta|\hat q_\alpha^s(0)|^2+\delta|q_\alpha(s)|^2+C_\delta|\nabla\hat q_\alpha^s(0)|^2+\delta|\nabla q_\alpha(s)|^2\dx\ds\\
					&=\delta 2T\|q_\alpha\|_{L^2(0,t;H^1)}^2+(C_\delta 2T+1)\int_0^t\int_0^s\|q_\alpha\|_{H^1}^2\,\mathrm d\tau\ds.
			\end{align*}
			This yields with \eqref{adjEst1} the crucial estimate
			\begin{align*}
				\|\hat q_\alpha^t\|_{L^2(0,t;H^1)}^2
					\leq{}&\delta\|q_\alpha\|_{L^2(0,t;H^1)}^2+C_\delta\int_0^t\|q_\alpha\|_{L^2(0,s;H^1)}^2.
			\end{align*}
			Analogously,
			\begin{align*}
				&\|\e(\hat p^t)\|_{L^2(0,t;L^2)}^2\leq \delta\|\e(p_\alpha)\|_{L^2(0,t;L^2)}^2+C_\delta\int_0^t\|\e(p_\alpha)\|_{L^2(0,s;L^2)}^2.
			\end{align*}
			By using these estimates and the estimates for $T_1$, ..., $T_6$ we obtain from \eqref{pqEst}
			\begin{align}
			\begin{aligned}
				&\|p_\alpha(t)\|_{L^2}^2+\|\e(p_\alpha)\|_{L^2(0,t;L^2)}^2+\|q_\alpha\|_{L^2(0,t;H^1)}^2\\
					&\qquad\leq
						C\int_0^t\Big(\|\e(p_\alpha)\|_{L^2(0,s;L^2)}^2+
						\|q_\alpha\|_{L^2(0,s;H^1)}^2\Big)\ds
						+C\|\chi_\alpha(T)-\chi_T\|_{L^2}^2.
			\end{aligned}
			\end{align}
			Thus Lemma \ref{lemmma:adjPQ} is proven after using Gronwall's lemma.
			\ep
		\end{proof}
		\vspace*{0.0em}\\
		The assertion of Proposition \ref{prop:adjointSystem} can now easily be obtained by exploiting the a priori estimate of the
		$\alpha$-regularized adjoint system from Lemma \ref{lemmma:adjPQ}.
		Due to the linearity of the PDE system \eqref{adjEq1}-\eqref{adjEq2} we can pass
		to the limit $\alpha\downarrow 0$.
		Thus Proposition \ref{prop:adjointSystem} is proven.\ep
	
	\subsection{Derivation of a first-order optimality system}
		This last part of the section is devoted to collect the results from below
		in order to prove our main result namely
		a necessary optimality system for minimizers of \textbf{(CP)}.
		For reader's convenience we summarize the approach to solve this problem.

		From now on we assume that \textbf{(A1)}-\textbf{(A5)}, \textbf{(O1)}-\textbf{(O3)}
		and \textbf{(B1)}-\textbf{(B4)} hold.
		Let us introduce the so-called ''reduced cost functional''
		given by
		\begin{align*}
			&j:\C B\to\R
				&&\text{ defined by }j(b):=\C J(\C S_{2}(b),b)\\
			&\hspace*{-2.1em}\text{with the cost functional}\phantom{\Big(}\\
			&\C J:\dX\times\C B\to\R
				&&\text{ defined by \eqref{functional}}\\
			&\hspace*{-2.1em}\text{and the control-to-state operator}\phantom{\Big(}\\
			&\C S:\C B\to\dU\times\dX
				&&\text{ defined by }\C S(b)=(\C S_1(b),\C S_2(b)):=(u(b),\chi(b))\\
				&&&\text{ solving PDE system \eqref{eqn:damageRegEq}-\eqref{eqn:initialRegEq}.}
		\end{align*}
		Our optimal control problem \textbf{(CP)} can now be restated as
		\begin{align*}
			\textbf{(CP')}\quad\text{find a minimizer of $j$ over $\C B_{adm}$.}
		\end{align*}
		Theorem \ref{ExistenceOptimalControl} guarantees existence of minimizers to \textbf{(CP)}.
		Let $b$ such a minimizer.
		We know that $\C J$ is Fr\'echet differentiable
		and from Proposition \ref{prop:strongDiff} that $\C S$ is G\^{a}teaux differentiable.
		Thus $j$ is also G\^{a}teaux differentiable.
		Since $\C B_{adm}$ is a bounded, closed and convex subset of $\C B$ (see \textbf{(B4)}),
		the desired necessary condition for optimality is 
		\begin{align}
			\langle D j(b),\widehat b -b\rangle_{\C B}\geq 0\quad\text{ for every }\widehat b \in\C B_{adm}.
		\label{chainRule2}
		\end{align}
		Application of the chain rule yields (see \eqref{chainRule} with $h=\widehat b-b$)
		\begin{align}
			\big\langle \partial_\chi\C J(\C S_2(b),b),D\C S_2(b)[\widehat b -b]\big\rangle_{\dX}
				+\big\langle \partial_b\C J(\C S_2(b),b),\widehat b -b\big\rangle_{\C B}\geq 0.
		\label{chainRule3}
		\end{align}
		Let $(p,q)$ be a solution of the adjoint problem at $(u,\chi):=\C S(b)$ according to
		Proposition \ref{prop:adjointSystem}.
		Testing equation \eqref{adjointSum} with the admissible pair of
		test-functions
		$$
			(\ph,\psi)=(D\C S_1(b)[\widehat b -b], D\C S_2(b)[\widehat b -b])=:(\du,\dchi)\in\dQ,
		$$
		we can rewrite the first term in \eqref{chainRule3} as the
		left-hand side of \eqref{adjointSum} tested with $(\du,\dchi)$.
		Then testing the PDE system for $(\du,\dchi)$ in Proposition
		\ref{prop:weakDiff} with the admissible pair of test-functions $(p,q)\in\aQ$ and
		adding the resulting equations
		we end up with (see \eqref{pqId}-\eqref{pqId2})
		$$
			\big\langle \partial_\chi\C J(\C S_2(b),b),D\C S_2(b)[\widehat b -b]\big\rangle_{\dX}
			=\int_\Sigma p\cdot(\widehat b -b)\dxt.
		$$
		Therefore \eqref{chainRule3} is equivalent to
		\begin{align}
			\int_\Sigma (p+\lambda_\Sigma b)\cdot(\widehat b -b)\dxt
				\geq 0\quad\text{ for every }\widehat b \in\C B_{adm}.
		\label{VI}
		\end{align}
		In conclusion we have proven the following result:
		\begin{theorem}
		\label{theorem:main}
			Suppose that \textbf{(A1)}-\textbf{(A5)}, \textbf{(O1)}-\textbf{(O3)}
			and \textbf{(B1)}-\textbf{(B4)} hold.
			Let $b\in\C B_{adm}$ be an optimal control of \textbf{(CP)} with the associated state
			$(u,\chi)=\C S(b)$ and adjoint variables $(p,q)\in\aQ$ that solves the system \eqref{adjP}-\eqref{adjQ} according to Proposition \ref{prop:adjointSystem}. Then \eqref{VI} holds.
		\end{theorem}
		
	\section{Conclusion and perspectives}
		In our preceding work \cite{FH15} we have proven well-posedness of strong solutions
		of the state system \eqref{eqn:elasticRegEq}-\eqref{eqn:initialRegEq}
		and existence of optimal controls for \textbf{(CP)}.
		Based on these results we have established first-order optimality conditions
		in this paper.
		The main result is stated in Theorem \ref{theorem:main} and provides a basis
		for further investigations.
		We conclude our paper with some open problems that could be addressed in future
		works.
		\begin{itemize}
			\item
				\textit{Irreversibility condition.}
					As pointed out in the introduction
					damage models usually contains a so-called irreversibility condition
					which is realized via the sub-differential term $\partial I_{(-\infty,0]}(\chi_t)$
					in the damage law. Equivalently we may introduce a slack variable $\zeta$
					and write the damage law as
					\begin{align*}
						\chi_t+\zeta-\Delta\chi_t-\Delta\chi+\frac 12 \mathbb C'(\chi)\e(u):\e(u)+f'(\chi)=0
					\end{align*}
					together with the complementarity conditions
					\begin{align*}
						\chi_t\cdot\zeta=0,\quad\zeta\geq 0,\quad\chi_t\leq 0.
					\end{align*}
					The corresponding optimal control problem then becomes a difficult and unexplored
					mathematical program with complementarity constraints (MPCC)
					and it remains open if stationarity conditions can be obtained
					via a limit passage of the regularized version as considered in this paper.
					As pointed out in \cite{Bar81} (see \cite{FH15} for our case)
					optima of the regularized control problem approximate solutions of the
					MPCC.
			\item
				\textit{Different cost functionals.}
					We have considered an $L^2$-tracking type cost functional
					in \textbf{(CP)} since this work is focused on the treatment of a complex and
					nonlinear state system.
					This restricts possible applications because (smooth approximations of) cracks only
					give rise to a small contribution with respect to the $L^2$-norm.
					More realistic choices would be the usage of higher-order or even $L^\infty$-cost functionals.
			\item
				\textit{Damage-dependent viscosities.}
					It would be desirable to let not only the stiffness tensor $\mathbb C(\cdot)$
					but also the viscosity tensor $\mathbb D$
					in the force balance equation \eqref{eqn:elasticRegEq}
					to depend on the damage phase-field $\chi$.
					Existence of strong solutions has already been proven in \cite{FH15}
					and well-posedness is also expected for this case.
					However optimality conditions for the optimal control problem still needs to be
					shown.
		\end{itemize}

	\addcontentsline{toc}{chapter}{Bibliography}{\footnotesize{\setlength{\baselineskip}{0.2 \baselineskip}
	\bibliography{references}
	}
	\bibliographystyle{plain}}
\end{document}